\documentclass[12pt]{amsart}

\usepackage{hyperref}
\hypersetup{nesting=true,debug=true,naturalnames=true}
\usepackage{graphicx,amssymb,upref,units}

\usepackage{tikz}

\let\<\langle
\let\>\rangle
\newcommand{\doi}[1]{doi:\,\href{http://dx.doi.org/#1}{#1}}

\let\uml\"
\newcommand{\mrev}[1]{\href{http://www.ams.org/mathscinet-getitem?mr=#1}{MR#1}}
\newcommand{\zbl}[1]{\href{http://www.emis.de/cgi-bin/MATH-item?#1}{Zbl #1}}

\usepackage{amsmath,amssymb,fullpage,xy,mathrsfs,amsthm,amsxtra,graphicx,verbatim}
\usepackage{hyperref}
\hypersetup{colorlinks=true,urlcolor=magenta,citecolor=magenta,linkcolor=blue}
\usepackage{enumerate}
\usepackage{colonequals}
\usepackage{multirow}

\usepackage[mathcal]{euscript}
\newcommand{\Gal}{\operatorname{Gal}}
\newcommand{\End}{\operatorname{End}}

\newcommand{\OK}{\mathcal{O}_{K}}

\newcommand{\Fq}{\mathbf{F}_{q}}
\newcommand{\Aut}{\operatorname{Aut}}

\newcommand{\Zl}{\mathbf{Z}_{\ell}}
\newcommand{\im}{\operatorname{im}}

\newcommand{\Z}{\mathbf{Z}}
\newcommand{\F}{\mathbf{F}}

\newcommand{\Q}{\mathbf{Q}}

\newcommand{\GL}{\operatorname{GL}}
\newcommand{\Mat}{\operatorname{M}}

\newcommand{\Frob}{\operatorname{Frob}}

\newcommand{\ddef}{\colonequals}
\newcommand{\disc}{\operatorname{disc}}

\newcommand{\tors}{\mathsf{tors}}

\newcommand{\sqf}{\mathsf{sqf}}

\DeclareMathOperator{\Prob}{Prob}

\numberwithin{equation}{subsection}

\theoremstyle{plain}
\newtheorem{thm}[equation]{Theorem}

\newtheorem{lem}[equation]{Lemma}
\newtheorem{defn}[equation]{Definition}
\newtheorem{cor}[equation]{Corollary}
\newtheorem{prop}[equation]{Proposition}

\newtheorem{conj}[equation]{Conjecture}

\theoremstyle{remark}
\newtheorem{rmk}[equation]{Remark}

\newtheorem{exm}[equation]{Example}

 \input xy
 \xyoption{all}

\begin{document}

\title{The probability of non-isomorphic group structures of isogenous elliptic curves in finite field extensions, I}

\author{John Cullinan}
\address{Department of Mathematics, Bard College, Annandale-On-Hudson, NY 12504, USA}
\email{cullinan@bard.edu}
\urladdr{\url{http://faculty.bard.edu/cullinan/}}

\author{Nathan Kaplan}
\address{Department of Mathematics, University of California, Irvine, CA 92697, USA}
\email{nckaplan@math.uci.edu }
\urladdr{\url{https://www.math.uci.edu/~nckaplan/}}

\begin{abstract}
Let $\ell$ be a prime number and let $E$ and $E'$ be $\ell$-isogenous elliptic curves defined over a finite field $k$ of characteristic $p \ne \ell$.  Suppose the groups $E(k)$ and $E'(k)$ are isomorphic, but $E(K) \not \simeq E'(K)$, where $K$ is an $\ell$-power extension of $k$.  In a previous work we have shown that, under mild rationality hypotheses, the case of interest is when $\ell=2$ and $K$ is the unique quadratic extension of $k$.
In this paper we study the likelihood of such an occurrence by fixing a pair of 2-isogenous elliptic curves $E$, $E'$ over $\Q$ and asking for the proportion of primes $p$ for which $E(\F_p) \simeq E'(\F_p)$ and $E(\F_{p^2}) \not \simeq E'(\F_{p^2})$.
\end{abstract}

\maketitle

\section{Introduction}

\subsection{Overview}

Let $E$ and $E'$ be elliptic curves defined over a finite field $k$.  If $E$ and $E'$ are isogenous, then the  groups $E(k)$ and $E'(k)$ have the same order, but might not be isomorphic. For example, if $k$ is the field of 5 elements and $E$ and $E'$ are defined by
\begin{align*}
E: y^2 &= x^3+x \\
E': y^2 &= x^3 + x +2,
\end{align*}
then $E(k) \simeq \Z/2\Z \times \Z/2\Z$ and $E'(k) \simeq \Z/4\Z$. However, if $E(k) \simeq E'(k)$, does that imply that $E(K) \simeq E'(K)$, as $K$ ranges over finite extensions of $k$?  It is similarly easy to see that the answer is no:

\begin{exm} \label{exm1}
Let $k$ be the field of 17 elements and $K$ the field of $17^2$ elements.  Let 
\begin{align*}
E: y^2 &= x(x+6)(x+12) \\
E': y^2 &= (x+1)(x+4)(x-4).
\end{align*} 
Observe that $E' = E/\langle (0,0) \rangle$, so $E$ and $E'$ are 2-isogenous.  One can check that
\[
E(k) \simeq E'(k) \simeq \Z/2\Z \times \Z/10\Z,
\]
but $E(K) = \Z/8\Z \times \Z/40\Z$ and $E'(K)= \Z/4\Z \times \Z/80\Z$.
\end{exm}

It is even possible to come up with examples where $E(K) \simeq E'(K)$ for all finite extensions $K/k$, but $E$ and $E'$ are not isomorphic as curves.  However, extreme examples such as these, and routine ones like \ref{exm1},  only happen under very specific circumstances.  In fact, Example \ref{exm1} can be viewed as a ``worst case scenario'' in the sense that, under mild rationality conditions, it is only in the context of rational 2-isogenies and quadratic extensions that we can have $E(k) \simeq E'(k)$ and $E(K) \not \simeq E'(K)$.  (We will explain these rationality assumptions in detail below.)  In this paper we aim to understand  \emph{how often} examples such as \ref{exm1} occur.

\subsection{Reduction to 2-isogenies}  Throughout this paper we consider the simplest case where $E$ and $E'$ are related by an isogeny of prime degree $\ell$, coprime to the characteristic of $k$.  If we additionally assume that $E(k)$ and $E'(k)$ have order divisible by $\ell$ and the kernel of the isogeny $E \to E'$ is generated by a rational point of order $\ell$, then we recall the main theorem of \cite{cullinan1}:

\begin{thm}[Theorem 1 of \cite{cullinan1}] \label{cull1thm}
Let $k$ be a finite field of odd characteristic $p$, $\ell \ne p$ an odd prime, and $E$ and $E'$ ordinary elliptic curves over $k$.  Suppose $E$ and $E'$
\begin{itemize}
\item each have a point of order $\ell$ defined over $k$, and 
\item are $\ell$-isogenous with kernel generated by a point defined over $k$.
\end{itemize}
Then $E(k) \simeq E'(k)$ if and only if $E(K) \simeq E'(K)$ for all finite extensions $K/k$.
\end{thm}

Concretely, \emph{when $\ell$ is odd}, if $E$ and $E'$ are $\ell$-isogenous and each has a non-trivial point of order $\ell$ defined over $k$, then $E(k) \simeq E'(k)$ implies $E(K) \simeq E'(K)$ in all finite extensions $K/k$.   

\bigskip

Because it is known how the $\ell$-part of the groups of rational points grow in towers \cite{lauter}, we can take from this theorem that the group structure of $E(k)$ completely determines the group structure of all $\ell$-isogenous curves to $E$, in all finite extensions $K/k$, under the hypothesis that the isogeny is generated by a $k$-rational point of order $\ell$.  (In case the isogeny has degree $\ell$, but the groups $E(k)$ and $E'(k)$ have order coprime to $\ell$, then one must perform an initial base-field extension to (say) $L$ to obtain a point of order $\ell$.  Then we can apply Theorem \ref{cull1thm} taking $L$ as the base field.)

Example \ref{exm1} shows that Theorem \ref{cull1thm} cannot be true when $\ell=2$.  But it also exemplifies the only way for Theorem \ref{cull1thm} to fail.  More precisely, in \cite{cullinan2} the first author proved the following theorem, showing exactly under which circumstances the groups of rational points of 2-isogenous curves fail to be isomorphic in towers:

\begin{thm}[Theorem 2 of \cite{cullinan2}] \label{cull2thm}
Let $E$ and $E'$ be ordinary, 2-isogenous elliptic curves defined over a finite field $k$ such that the isogeny is also defined over $k$. Suppose $E(k) \simeq E'(k)$.  Let the endomorphism ring of each curve be an order in the quadratic imaginary ring $\Z[\omega]$ and write $\pi = a+b\omega \in \Z[\omega]$, where $a$ is odd and $b$ is even, for the Frobenius endomorphism.  
Then $E(K) \simeq E'(K)$ for all finite extensions $K/k$ unless the following holds:
\[
v_2(a-1) = 1 \text{  and  } v_2(a+1) > v_2(b) - s_2.
\]
In that case, $E(K) \simeq E'(K)$ for odd-degree extensions $K/k$ only.
\end{thm}

\begin{rmk}
In the published version of \cite[Theorem 2]{cullinan2} there is a typographical error in the statement of the theorem.  There is an extra ``+1'' in the inequality for $v_2(a+1)$.  The corrected statement is listed above.
\end{rmk}

In the theorem, $s_2$ is a positive integer related to the conductors of the endomorphism rings of $E$ and $E'$. The upshot of this result is that there are precisely two possibilities. Either 
\begin{enumerate}
\item $E(K) \simeq E'(K)$ for all finite extensions $K/k$, or 
\item we can detect that $E(K) \not \simeq E'(K)$ in the unique quadratic extension $K/k$. Moreover, we can detect this failure by performing computations \emph{exclusively over $k$}.   
\end{enumerate}
We will review all of this background in detail in later sections of the paper. 

\subsection{Setup and Statement of the Main Results} Granting this background, we now set our notation and aims for the paper.  Let $E$ and $E'$ be 2-isogenous elliptic curves defined over a field $k$ such that the isogeny is also defined over $k$.  We call such a pair $(E,E')$ \textbf{rationally 2-isogenous}.  In this paper we  focus exclusively on the cases $k= \F_p$ and $k = \Q$. 

Fix an odd prime $p$. If $E$ and $E'$ are rationally 2-isogenous over $\F_p$, then $E(\F_p)$ has a point $P$ of order 2 and 
\[
E' = E/\langle P \rangle.
\]
We say that the pair $(E,E')$ is an  \textbf{anomalous pair} if $E$ and $E'$ are rationally 2-isogenous over $\F_p$, $E(\F_p) \simeq E'(\F_p)$, and  $E(\F_{p^2}) \not \simeq E'(\F_{p^2})$. As explained above, this is precisely the obstruction for rationally 2-isogenous curves having isomorphic group structures in towers over $\F_p$.

Here is the point of view we take for the paper.  Fix a pair of rationally 2-isogenous curves $(E,E')$ over $\Q$. We assume henceforth that $E$ and $E'$ do not have CM.  However, we will address the CM case in a forthcoming paper \cite{rnt}; see Section \ref{preview_2} for further details. To streamline notation, we will also use $E$ and $E'$ to denote the reductions modulo $p$ of the curves over $\Q$.   We call a prime $p$ of good reduction \textbf{anomalous} for $(E,E')$ if $E(\F_p) \simeq E'(\F_p)$ and $E(\F_{p^2}) \not \simeq E'(\F_{p^2})$.  Therefore, at an anomalous prime for the pair $(E,E')$ defined over $\Q$, we have that $(E,E')$ is an anomalous pair.  Depending on whether $p$ or $E$ is fixed, the two usages of ``anomalous'' should not be in conflict.   

Given this setup, we seek to understand the ratio
\begin{align} \label{anomlim}
\mathcal{P}(X)  = \frac{\#\lbrace \text{anomalous}~p \leq X \rbrace}{\pi(X)},
\end{align}
where $\pi(X)$ is the prime counting function, and also the limit $\mathcal{P} = \lim_{X \to \infty} \mathcal{P}(X)$, if it exists. We note that $\mathcal{P}(X)$ and $\mathcal{P}$ depend on both $E$ and $E'$ (more specifically, they depend on the images of the 2-adic representations over $\Q$ for each curve).  In this paper we only make one computation explicit: the case where the 2-adic images are isomorphic and as large as possible given the constraints of the setup.  The following examples show that there exist pairs $(E,E')$ for which anomalous primes exist, and there exist pairs for which they do not.  Throughout this paper when we refer to a proportion of primes with some property, or the probability that a prime has some property, we mean it is in this sense of counting primes up to $X$ and taking a limit.

\begin{exm}
Let $E$ be the elliptic curve \href{https://www.lmfdb.org/EllipticCurve/Q/210/e/5}{{\tt 210e5}} and $E'$ the curve \href{https://www.lmfdb.org/EllipticCurve/Q/210/e/4}{{\tt 210e4}} of the \texttt{LMFDB} \cite{lmfdb}.  Then $E$ and $E'$ are 2-isogenous, with $\Q$-torsion subgroups $\Z/2\Z \times \Z/8\Z$ and $\Z/2\Z \times \Z/4\Z$, respectively.  There are no anomalous primes for these curves, a consequence (as we will see below) of the sizes of $E(\Q)_\tors$ and $E'(\Q)_\tors$.   
\end{exm}

\begin{exm}
The isogeny class \href{https://www.lmfdb.org/EllipticCurve/Q/10608/y/}{{\tt 10608y}} consists of two elliptic curves, $E$ and $E'$, such that 
\[
E(\Q)  \simeq E'(\Q) \simeq \Z/2\Z;
\]
these are the smallest Mordell-Weil groups possible given that $E$ and $E'$ are rationally 2-isogenous.  Moreover, the mod 2 representation of each curve has order 2, and the 2-adic representation of each has index 3 in $\GL_2(\Z_2)$, \emph{i.e.},~is as large as possible given the hypotheses on each curve.     We consider all primes up to some bound and count those that are anomalous: 
\begin{center}
\begin{tabular}{|c|c|c|c|c|c|}
\hline
total $\#p$& $\# p \mid E(\F_p) \not \simeq E'(\F_p)$ & $\# p \mid E(\F_p) \simeq E'(\F_p)$ & \# anomalous \\
\hline
1000 &  539 & 457 & 30 \\
\hline
10000  & 5324 & 4672 & 335 \\
\hline
\end{tabular}
\end{center}
\end{exm}
Converting the number of good primes to a value of $X$, we see that 
\begin{align*}
\mathcal{P}(7919) &= \frac{30}{1000} \sim 0.03,\text{ and} \\
\mathcal{P}(104741) &= \frac{335}{10000} \sim 0.0335,
\end{align*}
suggesting that the limit might exist.

The main result of this paper is that the limit does exist and can be computed using the image of the 2-adic representations attached to $E$ and $E'$.  In general, a pair of rationally 2-isogenous elliptic curves define adjacent vertices on an {isogeny-torsion} graph over $\Q$.  In \cite{chil-alvaro} and \cite{chiloyan}, the authors give a classification of all isogeny-torsion graphs over $\Q$.  Moreover, the classification of Rouse and Zureick-Brown \cite{rzb} of the possible images of the 2-adic representation of elliptic curves over $\Q$ presents us with a finite list of graphs and images to consider for $E$ and $E'$.

In a forthcoming paper \cite{rnt} we work out, among other things,  the possible values that can occur for elliptic curves over $\Q$, including the CM case.  In this paper, we consider only one case and prove the following theorem.

\begin{thm} \label{mainthm2}
Let $E$ and $E'$ be rationally 2-isogenous elliptic curves over $\Q$ such that $[\GL_2(\Z_2): \im \rho_{E,2}] = [\GL_2(\Z_2):\im \rho_{E',2}] = 3$, \emph{i.e.},~both curves have maximal 2-adic image given that each has a rational 2-torsion point.  Then $\mathcal{P} = 1/30$.
\end{thm}

\begin{rmk}
The elliptic curves of Theorem \ref{mainthm2} are parameterized by the curve \href{http://users.wfu.edu/rouseja/2adic/X6.html}{$\mathtt{X_6}$} of the \href{http://users.wfu.edu/rouseja/2adic/}{\texttt{RZB}} database.
\end{rmk}

\begin{rmk}
See Section \ref{preview_2} for a discussion of the non-maximal cases and setup to be addressed in \cite{rnt}.
\end{rmk}

In order to get the result that $\mathcal{P} = 1/30$, we make full use of the structure of the \textbf{2-isogeny volcano} $V_p$ of $E$ at $p$.  The 2-isogeny volcano is a graph, the connected components of which consist of vertices (elliptic curves over $\F_p$) and edges ($\F_p$-rational 2-isogenies), that organizes the curves into levels (we reserve the term \emph{height} for the entire volcano and review our conventions in later sections).  All curves at the same level have isomorphic endomorphism rings, which implies that all curves at the same level have isomorphic group structures over $\F_p$.  

A 2-isogeny $E \to E'$ defined over $\F_p$ can be \textbf{vertical} ($[\End(E):\End(E')] = 2$ or $1/2$) or \textbf{horizontal} ($\End(E) \simeq \End(E')$); horizontal isogenies necessarily preserve the group structure over $\F_p$, while vertical isogenies may or may not.  At an anomalous prime $p$, we have the following confluence of events:
\begin{itemize}
\item the $\Q$-isogeny $E \to E'$ reduces to a vertical isogeny over $\F_p$, and
\item $E(\F_p)[2^\infty] \simeq E'(\F_p)[2^\infty] \simeq \Z/2\Z \times \Z/2\Z$ so that the volcano $V_p$ has the rough structure
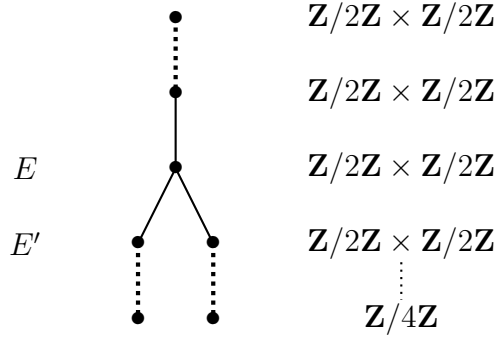
\begin{center}
\begin{tikzpicture}[scale=1.0,sizefont/.style={scale = 2}]
\draw[ultra thick] (1,5) node {${\bullet}$};
\draw[ultra thick, dotted] (1,5) -- (1,4);
\draw[ultra thick] (1,4) node {${\bullet}$};
\draw[thick] (1,3) -- (1,4);
\draw[ultra thick] (1,3) node {${\bullet}$};
\draw[ultra thick] (0.5,2) node {${\bullet}$};
\draw[ultra thick] (1.5,2) node {${\bullet}$};
\draw[ultra thick, dotted] (0.5,2) -- (0.5,1);
\draw[ultra thick, dotted] (1.5,2) -- (1.5,1);
\draw[ultra thick] (1.5,1) node {${\bullet}$};
\draw[ultra thick] (0.5,1) node {${\bullet}$};
\draw[thick] (1,3) -- (0.5,2);
\draw[thick] (1,3) -- (1.5,2);

\draw(4,5) node {$\Z/2\Z \times \Z/2\Z$};
\draw(4,4) node {$\Z/2\Z \times \Z/2\Z$};
\draw(4,3) node {$\Z/2\Z \times \Z/2\Z$};
\draw(4,2) node {$\Z/2\Z \times \Z/2\Z$};
\draw[thick, dotted] (4,1.75) -- (4,1.25);
\draw(4,1) node {$\Z/4\Z$};

\draw(-1,3) node {$E$};
\draw(-1,2) node {$E'$};

\end{tikzpicture}
\end{center}
where either $E$ or $E'$ lies at least two levels above the floor, and
\item $E(\F_{p^2})[2^\infty] \not \simeq E'(\F_{p^2})[2^\infty]$ and $E$ and $E'$ are situated on $V_{p^2}$ as follows

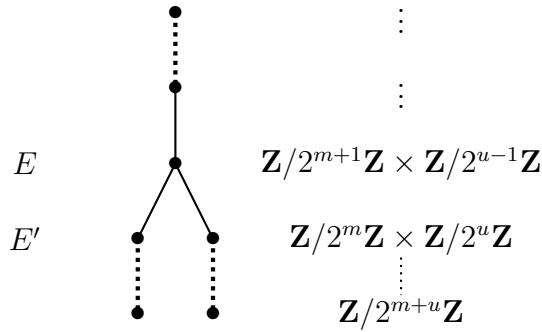
\begin{center}
\begin{tikzpicture}[scale=1.0,sizefont/.style={scale = 2}]
\draw[ultra thick] (1,5) node {${\bullet}$};
\draw[ultra thick, dotted] (1,5) -- (1,4);
\draw[ultra thick] (1,4) node {${\bullet}$};
\draw[thick] (1,3) -- (1,4);
\draw[ultra thick] (1,3) node {${\bullet}$};
\draw[ultra thick] (0.5,2) node {${\bullet}$};
\draw[ultra thick] (1.5,2) node {${\bullet}$};
\draw[ultra thick, dotted] (0.5,2) -- (0.5,1);
\draw[ultra thick, dotted] (1.5,2) -- (1.5,1);
\draw[ultra thick] (1.5,1) node {${\bullet}$};
\draw[ultra thick] (0.5,1) node {${\bullet}$};
\draw[thick] (1,3) -- (0.5,2);
\draw[thick] (1,3) -- (1.5,2);

\draw(4,5) node {$\vdots$};
\draw(4,4) node {$\vdots$};
\draw(4,3) node {$\Z/2^{m+1}\Z \times \Z/2^{u-1}\Z$};
\draw(4,2) node {$\Z/2^m\Z \times \Z/2^{u}\Z$};
\draw(4,1) node {$\Z/2^{m+u}\Z$};
\draw[thick, dotted] (4,1.75) -- (4,1.25);
\draw(-1,3) node {$E$};
\draw(-1,2) node {$E'$};

\end{tikzpicture}
\end{center}
\end{itemize}

We interpret the value $\mathcal{P} = 1/30$ as the sum of a geometric series, where the summands reflect the group theory of $\im \rho_{E,2}$ and $\im \rho_{E',2}$.  In particular, we filter the anomalous primes by \textbf{defect} (which we explain in detail in the sections below).  Briefly, an anomalous prime has defect $(a,b)$ if $E(\F_{p^2})$ has full $2^a$-torsion but not full $2^{a+1}$-torsion and $E(\F_{p^2})$ has full $2^b$-torsion, but not full $2^{b+1}$-torsion.  It is a fact about adjacent vertices on an isogeny volcano that a prime can only have defect $(m+1,m)$ or $(m,m+1)$.  (This is exemplified in the figure above.)  Filtering by defect, and weighting each defect by the size of the kernels of the homomorphisms $\im \overline{\rho}_{E,2^{m+1}} \to \im \overline{\rho}_{E,2^{m}}$ and $\im \overline{\rho}_{E',2^{m+1}} \to \im \overline{\rho}_{E',2^{m}}$, we obtain the summands in the geometric series.

To ease the cumbersome notation, we let $G = \im \rho_{E,2}$ and $G' = \im \rho_{E',2}$.  If $p$ is a good prime, let $F \in G$ and $F' \in G'$ denotes representatives of the class of Frobenius.  Note that even though as a quadratic irrational number $\pi = a+b\omega$ (the Frobenius endomorphism) is represented in $\End(E)$ and $\End(E')$ by the same integral expression, the interpretation in each ring is different when those rings are not isomorphic.  Given all of this, we prove the following finer version of Theorem \ref{mainthm2}.

\begin{thm} \label{equithm}      
Let $E$ and $E'$ be rationally 2-isogenous elliptic curves over $\Q$ such that $[\GL_2(\Z_2): G] = [\GL_2(\Z_2):G'] = 3$. Let $p$ be a prime such that $F \equiv -I \pmod{2^{m}}$ but $F\not\equiv-I \pmod{2^{m+1}}$.  Then with probability 1/2, $F' \equiv -I \pmod{2^{m}}$ and $p$ is not anomalous, and with probability 1/2, $F' \equiv -I \pmod{2^{m-1}}$ and $F'\not\equiv -I \pmod{2^m}$ and $p$ is anomalous of defect $(m+1,m)$.  Furthermore, this characterizes all anomalous primes of defect $(m+1,m)$.
\end{thm}

\begin{rmk}
A similar result holds for primes of defect $(m,m+1)$.
\end{rmk}

This brings us to the final portion of the paper where we re-interpret our results on anomalous primes and their defects in terms of a probabilistic model of the distribution of heights of volcanoes and the discriminants of the endomorphism rings at each level.

\subsection{Organization of the Paper}
In the next section we review background on elliptic curves over finite fields, in particular the relationship between the endomorphism ring and rational points.  We also recall the relevant history of this problem as well as the results in \cite{cullinan1} and \cite{cullinan2} that are applicable in this context.

As a rough guide to the results, the main point of Section \ref{general} is to determine the structure of the 2-Sylow subgroup of $E(\F_p)$ and $E'(\F_p)$ at anomalous primes.  This leads to the notion of the defect of an anomalous prime.

In Section \ref{Q} we prove Theorem \ref{mainthm2} by filtering the anomalous primes by defect, determining the exact proportion of each defect, and then summing over all defects.  We determine the exact proportion of each defect by re-interpreting the criteria of Section \ref{general} for a prime to be anomalous in terms of matrix conditions in the 2-adic representations attached to $E$ and $E'$.  Following this,  we interpret the defect of an anomalous prime as determining where on the isogeny volcano of the pair $(E,E')$ lies and give numerical data suggesting a finer relationship between anomalous primes and endomorphism rings.

Section \ref{preview_2} is dedicated to future work.  In particular, we contextualize the results of the present paper within the goals of a follow-up paper in which we explore the range of  values of $\mathcal{P}$ that can occur for elliptic curves over $\Q$, including CM curves.  

\subsection{Databases} We use two online databases in this work: the $L$-Functions and Modular Forms Database (\href{https://www.lmfdb.org/}{\texttt{LMFDB}}), and the classification of 2-adic images of Galois representations attached to elliptic curves over $\Q$, due to Rouse and Zureick-Brown (\href{http://users.wfu.edu/rouseja/2adic/}{\texttt{RZB}}), based on the paper \cite{rzb}.  Whenever we use an entry in the database, such as an isogeny class or elliptic curve in the \texttt{LMFDB}, or a modular curve in the \texttt{RZB} database, we link to that entry in the database.

\subsection{Notation} We will explain any specialized notation in main body of the paper, but we remind the reader of some standard conventions.  If $k$ is a field and $k^s$ a separable closure of $k$, then we write $\Gal_k$ for the Galois group of $k^s/k$.  If $E$ is an elliptic curve over $k$ and $\ell$ is a prime number, then we write $T_\ell E$ for the $\ell$-adic Tate module of $E$ and 
\begin{align*}
&\rho_{E,\ell}: \Gal_k \to \Aut (T_\ell E), \text{ and} \\
&\overline{\rho}_{E,\ell^n}: \Gal_k \to \Aut (T_\ell E \otimes \Z/\ell^n\Z) 
\end{align*}
for the $\ell$-adic and mod $\ell^n$ representations of $E$, respectively.  If $G \subseteq \GL_2(\Zl)$ is the image of the $\ell$-adic representation, then we write $G(\ell^n) \subseteq \GL_2(\Z/\ell^n\Z)$ for its reduction modulo $\ell^n$. 

If $R$ is a ring, then we write $M_n(R)$ for the ring of $n\times n$ matrices with entries in $R$.  Finally, if $p$ is a prime number, then we write $v_p: \Q^\times \to \Z$ for the $p$-adic valuation. 

\subsection{Acknowledgments} We would like to thank Andrew Sutherland for supplying us with the initial computations that suggested the correct value of $\mathcal{P}$.  The second author was supported by NSF Grants DMS 1802281 and DMS 2154223.

\section{Elliptic Curves over Finite Fields} \label{background}

\subsection{Endomorphism Rings and Rational Points} \label{first_background} Let $q$ be a power of an odd prime $p$ and let $E$ be an ordinary elliptic curve defined over $\Fq$; we will address supersingular curves in Section \ref{supsect}. Since $E$ is ordinary, its endomorphism ring $\End(E)$ is isomorphic to an order $\mathcal{O}$ in an imaginary quadratic field $K = \Q(\sqrt{D})$ for a squarefree negative integer $D$, and all endomorphisms of $E$ are defined over $\F_q$.  

Let $\OK$ denote the ring of integers of $K$.  Write $d_K$ for the discriminant of $\OK$, the maximal order of $K$.  Then 
\[
d_K = \begin{cases}
4D & \text{ if } D \equiv 2,3\pmod{4} \\
D & \text{ if } D \equiv 1 \pmod{4}.
\end{cases}
\]
Recall that if $g$ is a positive integer, then we denote by $\mathcal{O}_g:=\Z \oplus \Z g\omega$  the order of conductor $g$ in $\OK$, where 
\[
\omega = \begin{cases}
 (1+\sqrt{D})/2 & \text{if } D \equiv 2,3\pmod{4}\\
 \sqrt{D} & \text{if } D \equiv 1 \pmod{4}.
 \end{cases}
\] 
We may therefore write $\Z[\pi] = \mathcal{O}_f$ and $\OK = \mathcal{O}_1$.  Since $\End(E) = \mathcal{O}$ contains $\Z[\pi]$, we may write $\mathcal{O} = \mathcal{O}_g$ for some $g \mid f$ with 
\[
\mathcal{O}_f \subseteq \mathcal{O}_g \subseteq \mathcal{O}_1.
\]
If $\Delta_g$ denotes the discriminant of $\mathcal{O}_g$, then $\Delta_g = g^2d_K$. 

Identifying $\End(E)$ with an order in $\OK$, we may write the Frobenius endomorphism $\pi \in \End(E)$ explicitly as an element of $\OK$.  We now review how to do this.  Recall the well-known formulas relating the cardinality of $E(\Fq)$, the fundamental discriminant of $K$, and the trace $t$ of $\pi$:
\begin{align}
\#E(\Fq) &= 1 + q - t \\
4q &= t^2 - \beta^2\Delta_g,
\end{align}
where $t$ is the trace of Frobenius, $\beta$ is a positive integer, and $\Delta_g = g^2d_K$, as above.

Then $\pi$ has a unique integral representation $\pi = a+b\omega \in \Z[\omega]$ given by 
\begin{align*}
a &= \begin{cases} t/2 & \text{ if } D\equiv 2,3\pmod{4} \\ (t-\beta g)/2 & \text{ if }D \equiv 1\pmod{4}, \end{cases} \\
b&= \beta g.
\end{align*}
We also recall a fundamental result of Lenstra \cite{lenstra}, which gives the structure of $E(\F_{q^m})$ for all positive integers $m$:
\begin{equation}\label{eq_lenstra}
E(\F_{q^m}) \simeq \frac{\mathcal{O}}{(\pi^m - 1)}.
\end{equation}

\subsection{Isogenies} Keeping with the notation above, suppose that $E$ and $E'$ are isogenous (ordinary) elliptic curves defined over $\Fq$.  Then the groups $E(\F_q)$ and $E'(\F_q)$ have the same cardinality, as do the groups $E(\F_{q^m})$ and $E'(\F_{q^m})$, for all positive integers $m$. 

Let $\ell \ne p$ be a prime number.  If $E$ and $E'$ have endomorphism rings $\mathcal{O}$ and $\mathcal{O}'$, respectively, and are $\ell$-isogenous, then by a result of Kohel \cite[Prop.~21]{kohel} we have 
\[
[\mathcal{O}:\mathcal{O}'] = \ell, \ell^{-1}, \text{ or } 1.
\]
In the first two cases, the isogeny is called \emph{vertical} (ascending/descending, depending on the inclusion) and in the latter it is \emph{horizontal}.

Isogenous elliptic curves have the same trace of Frobenius.  In the case of a vertical isogeny, $\mathcal{O}$ and $\mathcal{O}'$ are orders in $\OK$ of relative index $\ell$.  We explain what happens when $\mathcal{O}' \subseteq \mathcal{O}$.  (There is a completely analogous setup when $\mathcal{O} \subseteq \mathcal{O}'$.)  There exist divisors $g$ and $g'$ of $f$ such that $g'/g = \ell$ and $\mathcal{O} = \mathcal{O}_g$, $\mathcal{O}' = \mathcal{O}_{g'}$ with  
$$
\Z[\pi] = \mathcal{O}_f \subseteq \mathcal{O}_{g'} \subseteq \mathcal{O}_{g} \subseteq \mathcal{O}_1 = \OK.
$$

Turning to the group structures of isogenous curves, we recall that the main results of \cite{heuberger} and \cite{wittmann}  give criteria for any pair of  isogenous elliptic curves to have isomorphic groups of $\F_{q^{m}}$-rational points in terms of the prime divisors of the integral components of $\pi^m$.  We now recall some of the special notation introduced in \cite{heuberger} that we will adopt throughout the rest of this paper. 

Define a finite set of prime numbers $\mathbf{P}$ as follows, incorporating the notation above:
\[
\mathbf{P} = \lbrace p\text{ prime} \mid v_p(g) \ne v_p(g') \rbrace.
\]
For each $p \in \mathbf{P}$ we set 
\[
s_p  = \max \lbrace v_p(g), v_p(g') \rbrace,
\]
whence $s_p \geq 1$.   With this notation in place, write 
\[
\pi^m = a_m + b_m\omega,
\] 
for integers $a_m,b_m$.  Finally, we recall the criterion of \cite[Thm.~2.4]{heuberger} for $E(\F_{q^m})$ and $E'(\F_{q^m})$ to be isomorphic:
\begin{align} \label{heuberger_criterion}
E(\F_{q^m}) \simeq E'(\F_{q^m}) \Longleftrightarrow v_p(a_m-1) \leq v_p(b_m) - s_p,
\end{align}
for all $p \in \mathbf{P}$. 

Now we specialize to the situation that is the primary focus of this paper.  When the degree of the vertical isogeny $E \to  E'$ is a prime number $\ell$, then $g'/g = \ell^{\pm 1}$ and so $\mathbf{P} = \lbrace \ell \rbrace$.  For descending isogenies we have $v_\ell(g') = 1 + v_\ell(g)$ and for ascending isogenies we have $v_\ell(g) = 1+ v_\ell(g')$.   Specializing further, we set $\ell=2$ for the remainder of the paper.  In \cite[Thm.~2]{cullinan2} the first author proved that if $E(\Fq) \simeq E'(\Fq)$ and $E(\F_{q^2}) \simeq E'(\F_{q^2})$, then $E(\F_{q^m}) \simeq E(\F_{q^m})$ for all positive integers $m$.  Theorem \ref{cull2thm} gives the precise conditions under which the second isomorphism fails, given the first.  

\subsection{Supersingular Curves} \label{supsect}

In the case where $E$ and $E'$ are supersingular curves over $\F_p$ the situation is (perhaps surprisingly)  much simpler.  We recall the following result of Wittmann.

\begin{thm}[Theorem 4.1 of \cite{wittmann}] \label{wittmann1}
Let $E/\F_p$ be a supersingular elliptic curve.  Then 
$$
E(\F_{p^{2k}}) \simeq \Z/((-p)^k-1)\Z \times \Z/((-p)^k-1)\Z. 
$$
Further:
\begin{itemize}
\item If $p \not \equiv 3 \pmod{4}$ or $p \equiv 3 \pmod{4}$ and $E[2] \not \subseteq E(\F_p)$ we have 
$$
E(\F_{p^{2k+1}}) \simeq \Z/(p^{2k+1} + 1)\Z \text{ and } \End_{\F_p}(E) \simeq \Z[\sqrt{-p}].
$$
\item If $p \equiv 3 \pmod{4}$ and $E[2] \subseteq E(\F_p)$ we have
$$
E(\F_{p^{2k+1}}) \simeq \Z/2\Z \times \Z/\left( \frac{p^{2k+1}+1}{2} \right)\Z \text{ and } \End_{\F_p}(E) \simeq \Z[(1+\sqrt{-p})/2].
$$
\end{itemize}
\end{thm}

In \cite{cullinan2} we observed that this immediately implies that when when $E$ and $E'$ are supersingular, then the group structure over $\F_p$ completely determines the group structure over any finite extension:

\begin{cor}[Corollary 1 of \cite{cullinan2}] \label{cor1}
Let $p$ be a prime.  Let $E_1$ and $E_2$ be supersingular, isogenous elliptic curves defined over $\F_p$.  Suppose $E_1(\F_p) \simeq E_2(\F_p)$.  Then $E_1(K) \simeq E_2(K)$ for every finite extension $K/\F_p$.
\end{cor}

\section{General Properties of Anomalous Primes and Curves} \label{general}

We retain the notation and setup of the previous sections, in particular we assume $E$ and $E'$ are ordinary.  We start with a general property of anomalous pairs.

\begin{prop}
Let $(E,E')$ be an anomalous pair of elliptic curves defined over the finite field $\F_p$.  Then $p \equiv 1 \pmod{4}$.
\end{prop}

\begin{proof} 
Suppose $p \equiv 3 \pmod{4}$.  We distinguish between the cases where $|E(\F_p)| \equiv 2\pmod{4}$ versus $|E(\F_p)| \equiv 0 \pmod{4}$. Recall that if $(E,E')$ is an anomalous pair then in the representation $\pi = a+b\omega$ of Frobenius as an element of $\OK$, we have that $b$ is even; write $b=2b'$.

If $|E(\F_{p})| \equiv 0 \pmod{4}$, then $t \equiv 0 \pmod{4}$; write $t = 4t'$.  Since 
\[
4p = t^2 - b^2d_K = 16(t')^2 - 4(b')^2d_K,
\]
we must have $p = 4(t')^2 - (b')^2d_K$.  Thus $b'$ and $d_K$ are odd.  In particular, $v_2(b) =1$.  But since $(E,E')$ is an anomalous pair, we have 
\[
v_2(a-1) = 1 \leq v_2(b) - s_2 = 1 - s_2,
\]
whence $s_2 =0$.  But this means $\End(E) \simeq \End(E')$, contradicting the fact that $(E,E')$ are an anomalous pair.

If $|E(\F_p)| \equiv 2 \pmod{4}$, then $t \equiv 2 \pmod{4}$, so write $t = 2t'$ with $t'$ odd.  But then 
\[
p  = (t')^2 - (b')^2d_K.
\]
Since $p \equiv 3 \pmod{4}$ and $(t')^2 \equiv 1 \pmod{4}$, we must have $(b')^2d_K \equiv 2 \pmod{4}$.  But since $b'$ is odd and $d_K \equiv 0$ or $1 \pmod{4}$, this is impossible.  We conclude that if $p \equiv 3 \pmod{4}$ then $p$ cannot be anomalous. 
\end{proof}

\begin{lem} \label{2torspt}
If $|E(\F_p)| \equiv 2 \pmod{4}$ then $E(\F_p) \simeq E'(\F_p)$.
\end{lem}

\begin{proof}
Since $E$ and $E'$ are 2-isogenous, the prime-to-2 parts of $E(\F_p)$ and $E'(\F_p)$ are isomorphic \cite[Cor.~3]{cullinan1}.  Since each has a single point of order 2, the result follows by the structure theorem for finite abelian groups.
\end{proof}

\begin{thm}
If $|E(\F_p)| \equiv 2 \pmod{4}$ then $E(\F_{p^2}) \simeq E'(\F_{p^2})$.
\end{thm}

\begin{proof}
If $|E(\F_p)| \equiv 2 \pmod{4}$ then by  Lemma \ref{2torspt} we have $E(\F_p) \simeq E'(\F_p)$. If, in addition, 
$E(\F_{p^2}) \not \simeq E'(\F_{p^2})$, then $(E,E')$ is anomalous whence $p \equiv 1 \pmod{4}$.  Writing $\pi = a + b\omega$ in the notation of Section 2, we have 
\begin{enumerate}
\item $v_{2}(a-1) = 1 \leq v_2(b) - s_2$, \text{ and} 
\item $v_2(a+1) >  v_2(b) - s_2$.
\end{enumerate}
Since $v_2(a-1) = 1$, we have $a \equiv 3 \pmod{4}$.  We also have 
\begin{align} \label{n1_is_2_mod_4}
|E(\F_p)| = 1 + p - t \equiv 2\pmod{4},
\end{align}
hence $t \equiv 0 \pmod{4}$.  Now we divide the argument into two cases based on $D \pmod{4}$, where $D$ is the squarefree integer for which $K = \Q(\sqrt{D})$ is the endomorphism algebra of $E$ (and $E'$).

If $D \equiv 2,3\pmod{4}$, then $a=t/2$ and so $t \equiv 6\pmod{8}$, a contradiction.  If $D \equiv 1 \pmod{4}$, then we first recall the inequality (1).  Since $(E,E')$ is an anomalous pair, we must have $s_2 \geq1$ (otherwise, $\End(E) \simeq \End(E')$), and so we conclude that $v_2(b) \geq 2$.  But when $D \equiv 1 \pmod{4}$, we have $a = (t-b)/2$. Since both $t$ and $b$ must be divisible by 4, we get that $a$ is even.  This contradicts $a \equiv 3 \pmod{4}$, established above. 
\end{proof}

\begin{cor} \label{2mod4cor}
If $|E(\F_p)| \equiv 2 \pmod{4}$ then $E(\F_{p^m}) \simeq E'(\F_{p^m})$ for all positive integers $m$.
\end{cor}

\begin{proof}
This follows from \cite[Thm.~2]{cullinan2}: if $E(\F_{p^m}) \simeq E'(\F_{p^m})$ for $m \in \lbrace 1,2 \rbrace$, then $E(\F_{p^m}) \simeq E'(\F_{p^m})$ for all positive integers $m$. 
\end{proof}

Therefore, \emph{every} pair of curves $E,E'$ over $\F_p$ with $|E(\F_p)| \equiv 2 \pmod{4}$  and that are rationally 2-isogenous have isomorphic Mordell-Weil groups in all finite extensions.  Therefore, any anomalous pair must have $|E(\F_p)| \equiv 0 \pmod{4}$ and $p \equiv 1 \pmod{4}$.  

Next, we define a finer notion of $(E,E')$ being an anomalous pair.  This will carry over to a refined notion of $p$ being an anomalous prime, which will be an important topic in the following sections. Because the 2-Sylow subgroups of $E(\F_{p^2})$ and $E'(\F_{p^2})$ have the same size, but are not isomorphic, we can ask how they differ.  We describe this difference using the notion of defect.

\begin{defn}
Let $E \to E'$ be rationally 2-isogenous elliptic curves over $\Q$ and let $p$ be an anomalous prime.  If 
\begin{align*}
a  &= \max \lbrace i \in \mathbf{N} ~|~ E(\F_{p^2})[2^\infty] \supseteq \Z/2^{i}\Z \times \Z/2\Z^{i} \rbrace,\text{ and} \\
a'  &= \max \lbrace i \in \mathbf{N} ~|~ E'(\F_{p^2})[2^\infty] \supseteq \Z/2^{i}\Z \times 2^{i}\Z \rbrace,
\end{align*}
then we say that $p$ has \textbf{defect} $(a,a')$.
\end{defn}

\begin{rmk}
It is a well-known property of the \textbf{$\ell$-isogeny volcano} (which we will recall in Section \ref{volcanoes}) that if $E$ and $E'$ are $\ell$-isogenous elliptic curves over a finite field $k$ and the $\ell$-Sylow subgroups of $E(k)$ and $E'(k)$ are not isomorphic, then $E(k)[\ell^\infty] \simeq \Z/\ell^u\Z \times \Z/\ell^v\Z$ and $E'(k)[\ell^\infty] \simeq \Z/\ell^{u-1}\Z \times \Z/\ell^{v+1}\Z$ or $E'(k)[\ell^\infty] \simeq \Z/\ell^{u+1}\Z \times \Z/\ell^{v-1}\Z$ for some positive integer $u$ and nonnegative integer $v$.  Theorem \ref{establish_defect} establishes a similar result and relates the defect of an anomalous prime to the 2-valuation of the Frobenius endomorphism. 
\end{rmk}

We now make an observation concerning the 2-Sylow subgroups of anomalous pairs.

\begin{lem} \label{22lem}
Suppose $(E,E')$ is an anomalous pair.  Then $E(\F_p)[2^\infty] \simeq  E'(\F_p)[2^\infty] \simeq \Z/2\Z \times \Z/2\Z$.  
\end{lem}

\begin{proof}
If $(E,E')$ is an anomalous pair, then we must have $p \equiv 1 \pmod{4}$ and $|E(\F_p)| \equiv 0 \pmod{4}$, as previously established.  If neither curve has full 2-torsion defined over $\F_p$, then the 2-Sylow subgroups of $E(\F_p)$ and $E'(\F_p)$ are cyclic and the curves are rationally 2-isogenous.  By \cite[Thm.~1.2]{aw}, this is not possible.  This establishes that $E(\F_p)[2] \simeq E'(\F_p)[2] \simeq \Z/2\Z \times \Z/2\Z$.

To see that $E(\F_p)[2^\infty] \simeq  E'(\F_p)[2^\infty] \simeq \Z/2\Z \times \Z/2\Z$ as well, recall from \cite[p.~742]{miret2} that if $E$ and $E'$ are 2-isogenous and have isomorphic group structures over $\F_p$, then it must be the case that $E(\F_p)[2^\infty] \simeq  E'(\F_p)[2^\infty] \simeq \Z/2^k\Z \times \Z/2^k\Z$ for some $k$, hence $|E(\F_p)| = p+1-t \equiv 0 \pmod{2^{2k}}$.  Suppose $k>1$.  Then both curves will have at least full $2^{k+1}$-torsion over $\F_{p^2}$, and at least one will have full $2^{k+2}$-torsion (since $E(\F_{p^2}) \not \simeq E'(\F_{p^2})$).  Therefore ,
\[
|E(\F_{p^2})| = (p+1-t)(p+1+t) \equiv 0 \pmod{2^{2k+4}},
\]
and so $p+1+t \equiv 0 \pmod{16}$.  Since $k>1$, we have $p+1-t \equiv 0 \pmod{16}$ as well, which implies that $t \equiv 0 \pmod{8}$.  But this contradicts the fact that for an anomalous pair we must have $t \equiv 2 \pmod{4}$.  This completes the proof.
\end{proof}

We will apply the following result in the proof of Theorem \ref{establish_defect}.
\begin{lem} \label{2_torsion_frob}
Let $E$ be an ordinary elliptic curve defined over a finite field $\Fq$ of odd characteristic.  Let $\pi \in \End(E)$ be the Frobenius endomorphism.  If $v$ is the largest integer such that $\pi^m -1$ factors as $2^v\alpha$ in $\End(E)$, then $E(\F_{q^m})$ has full $2^v$-torsion but not full $2^{v+1}$-torsion. 
\end{lem}

\begin{proof}
By Lenstra's  theorem \eqref{eq_lenstra} \cite[Thm.~1(a)]{lenstra}, we have $E(\F_{q^m}) \simeq \End(E)/(\pi^m-1)$.   If $\pi^m-1$ factors as $2^v\alpha$, then clearly $E(\F_{q^m})$ has full $2^v$-torsion. By factoring isogenies via \cite[Thm.~25.1.2]{galbraith}, any $\F_q$-rational endomorphism of $E$ whose kernel contains the $2^{v+1}$-torsion points would have to factor as $2^{v+1}\beta$ in $\End(E)$. Thus $E(\F_{q^m})$ has full $2^v$-torsion but not full $2^{v+1}$-torsion.  
\end{proof}

\begin{thm} \label{establish_defect}
Let $E \to E'$ be 2-isogenous elliptic curves over $\Q$ and let $p$ be an anomalous prime.  Suppose $\End(E) = \mathcal{O}_g$ and $\End(E) = \mathcal{O}_{g'}$ are orders of conductor $g$ and $g'$, respectively, in the the imaginary number ring $\OK = \Z + \Z\omega$; write $\pi = a + b\omega$ with $b = \beta g = \beta'g'$.  Then $p$ has defect $(m+1,m)$ or $(m,m+1)$ for some integer $m \geq 2$, where $m = v_2(\beta)$.
\end{thm}

\begin{proof}
The isogeny $E \to E'$, initially defined over $\Q$, reduces modulo $p$ to a vertical isogeny (if the reduction were horizontal then $\mathcal{O}_g = \mathcal{O}_{g'}$ and $p$ would not be anomalous).  For the remainder of the proof we assume the isogeny is descending and will conclude that $p$ has defect $(m+1,m)$; an identical argument for ascending isogenies would show that $p$ has defect $(m,m+1)$.  

Write $\End(E) = \mathcal{O}_g = \Z + g\Z\omega$ and $\End(E') = \mathcal{O}_{g'} = \Z + g'\Z\omega$.  We have $g' = 2g$ and also write $\pi = a+ b\omega$ with $b = \beta g$ as established in Section \ref{first_background}. Since $p$ is anomalous and since $v_2(g') = v_{2}(g) + 1$, we have 
\[
v_2(a-1) = 1 \leq v_2(b)-s_2 = v_{2}(\beta) - 1 < v_2(a+1).
\]
Observe that $v_2(\beta) \geq 2$.

Now we compute 
\[
\pi^2 -1 = \begin{cases} (a^2 - 1 + b^2D) + 2ab\omega & \text{if $d_K = 4D$ and $D \equiv 2,3\pmod{4}$, and} \\
(a^2-1 + b^2 \left(\frac{D-1}{4} \right)) + (2ab+b^2)\omega & \text{if $d_K = D$ with $D \equiv 1\pmod{4}$}.\end{cases}
\]

In $\mathcal{O}_g$ we can factor, 
\[
\pi^2 - 1 =
\begin{cases}
&(a^2-1+\beta^2g^2D) + (2\beta) a g\omega,\text{ or}\\
&(a^2-1 + \beta^2g^2 \left(\frac{D-1}{4} \right)) + 2\beta (a + (\beta/2))g\omega,
\end{cases}
\]
depending on $d_K \pmod{4}$. In the first case (since $a$ is odd) and in the second case (since $a$ is odd and $\beta/2$ is even), $\pi^2 - 1$ is divisible in $\mathcal{O}_g$ by $2^{v_2(\beta)+1}$ and no higher power of 2.

Similarly, in $\mathcal{O}_{g'} = \mathcal{O}_{2g}$, $\pi^2-1$ is divisible by $2^{v_2(\beta)}$ and no higher power of 2.  By Lemma \ref{2_torsion_frob}, $E(\F_{p^2})$ has full $2^{v_{2}(\beta)+1}$-torsion (and no higher) and $E'(\F_{p^2})$ has full $2^{v_2(\beta)}$-torsion (and no higher).  Thus $p$ has defect $(m+1,m)$ for some integer $m = v_2(\beta) \geq 2$.  
\end{proof}

In the next section we interpret anomalous primes and their defects in relation to isogeny volcanoes.

\section{Isogeny Volcanoes of Elliptic Curves} \label{volcanoes}

Following a brief recap of the theory of isogeny volcanoes of ordinary elliptic curves, our purpose in this section is to prove a key proposition in service of Theorems \ref{mainthm2} and \ref{equithm}. We do not intend for this to be a complete treatment of the background material; we refer the reader to \cite{sutherland_volcano} for further details and proofs.

Let $q$ be a power of a prime $p$ and $E$ an ordinary elliptic curve over $\Fq$.  Let $V_q$ be the connected component of the 2-isogeny graph (volcano) containing $E$.  Then $V_q$ is a graph whose vertices correspond to elliptic curves defined over $\Fq$ that are 2-power $\Fq$-rationally isogenous to $E$ and edges are $\Fq$-rational 2-isogenies.  Thus, in our setup, $E$ and $E'$ represent adjacent vertices on the graph $V_p$; note that $V_p$ is a subgraph of $V_{p^2}$. 

Let $q$ be a power of $p$ and $T$ the trace of Frobenius over $\Fq$.  Let $\sqf(m)$ denote the squarefee part of an integer $m$.  where $\mathcal{O}_0$ is the endomorphism ring of an elliptic curve lying on the crater of $V_q$. Let $K = \Q(\sqrt{T^2-4q}) = \Q(\sqrt{D})$ where $D = \sqf(T^2-4q)$.  Then 
\[
\disc \mathcal{O}_K = \begin{cases} D & \text{ if } D \equiv 1 \pmod{4}, \text{ and} \\ 4D & \text{ if } D \equiv 2,3\pmod{4}. \end{cases}
\]
A theorem of Kohel \cite[Theorem 7(5)]{sutherland_volcano} shows that for a 2-isogeny volcano $2 \nmid [\mathcal{O}_K\colon \mathcal{O}_0]$.  

The \emph{height} of the volcano $V_q$ is given by \cite[Thm.~7]{sutherland_volcano}
\begin{equation}\label{eqn:volcano_height}
h(V_q) = \frac{1}{2} v_2 \left( \frac{T^2 - 4q}{\disc \mathcal{O}_0} \right) = \frac{1}{2} v_2 \left( \frac{T^2 - 4q}{\disc \mathcal{O}_K} \right).
\end{equation}
We choose the opposite labeling of the height as defined in \cite{sutherland_volcano} (there it is called the \textbf{depth}) and declare the floor of the volcano to have height 0.  In the case that $V_q$ consists of an isolated vertex, we set $h(V_q)=0$.  The subgraph of vertices at level $h(V_q)$ is called the \textbf{crater} of the volcano.    This labeling is more convenient for interpreting the defect of an anomalous prime in terms of the location of $E$ and $E'$.

The endomorphism rings of the elliptic curves at the same level of the volcano are isomorphic, hence the 2-Sylow subgroups at the same level are isomorphic.  Elliptic curves on the floor of a volcano have cyclic 2-Sylow subgroups \cite[\S3]{sutherland_volcano}, say of order $2^\nu$.  Then, for each $0 < m \leq  h_{\rm stab}$, we have the 2-Sylow subgroup at height $m$ is $\Z/2^{m}\Z \times \Z/2^{\nu - m}\Z$.  If $h_{\rm stab} < h(V_p)$ then the volcano is called \textbf{irregular} and $h_{\rm stab}$ is called the \textbf{stability level} \cite[p.~742]{miret2}. By \cite[\S4]{miret2}, all curves between the stability level and the crater have isomorphic 2-Sylow subgroups. We refer to the levels of the volcano between the stability level and the crater as the \textbf{stability zone}. 

\begin{lem} \label{heightlem}
Let $E$ and $E'$ be 2-isogenous elliptic curves defined over $\F_p$.  Let $V_p$ be the isogeny volcano which contains $E$ and $E'$ as adjacent vertices; let $V_{p^2}$ be the isogeny volcano over $\F_{p^2}$.  Suppose $t\equiv 2\pmod{4}$.  Then $h(V_{p^2}) = h(V_p)+1$.
\end{lem}

\begin{proof}
Let $t$ be the trace of $\pi_E$ and $T$ the trace of $\pi_E^2$.  By assumption $v_2(t)=1$. We have $T = t^2 - 2p$ since $|E(\F_{p^2})| = (p+1-t)(p+1+t)$.  Then
\begin{align*}
h(V_{p^2}) &=  \frac{1}{2} v_2 \left( \frac{T^2 - 4p^2}{\disc \mathcal{O}_0} \right) 
= \frac{1}{2} v_2 \left( \frac{(T-2p)(T+2p)}{\disc \mathcal{O}_0} \right)
= \frac{1}{2} v_2 \left( \frac{t^2 - 4p}{\disc \mathcal{O}_0} \, t^2 \right) 
= h(V_p) + 1.
\end{align*}
\end{proof}

\begin{rmk} The hypothesis that $t \equiv 2 \pmod{4}$ means that this lemma will be applicable to the case of anomalous pairs of elliptic curves.
\end{rmk}

\begin{prop}
Let $E$ and $E'$ be 2-isogenous elliptic curves defined over a finite field $\F_p$ and suppose $(E,E')$ is an anomalous pair.  Then:
\begin{itemize}
\item $V_p$ is irregular, and
\item $E$ and $E'$ represent adjacent edges on $V_p$ in the stability zone, and
\item $E$ and $E'$ do not both lie in the stability zone on $V_{p^2}$.
\end{itemize}
\end{prop}

\begin{proof}
This is just a matter of terminology.  Since $(E,E')$ is an anomalous pair, they are vertically isogenous.  By Lemma \ref{22lem}, we have $E(\F_p)[2^\infty] \simeq E'(\F_p)[2^\infty] \simeq \Z/2\Z \times \Z/2\Z$, hence neither curve lies on the floor of the volcano $V_p$.  Since the 2-Sylow subgroups are isomorphic, $V_p$ is an irregular volcano and the curves must lie in the stability zone.  However, over $\F_{p^2}$ the 2-Sylow subgroups are not isomorphic, hence at least one curve lies outside the stability zone.
\end{proof}

Note that since $\disc \mathcal{O}_0 = \disc \mathcal{O}_K [\mathcal{O}_K \colon \mathcal{O}_0]^2$ and $[\mathcal{O}_K \colon \mathcal{O}_0]$ is odd, we have that $\disc \mathcal{O}_0 \equiv \disc \mathcal{O}_K \pmod{8}$.  Turning now to the endomorphism rings, we distinguish between the congruence classes $\disc \mathcal{O}_0 \equiv 0,1,4,5 \pmod{8}$.  In these cases, the shape of the crater corresponds to the discriminant in the following way, as established by \cite[Thm.~7]{sutherland_volcano}.  When $\disc \mathcal{O}_0 \equiv 0 \pmod{4}$ or $\disc \mathcal{O}_0 \equiv 5 \pmod{8}$ then the volcanoes have shapes
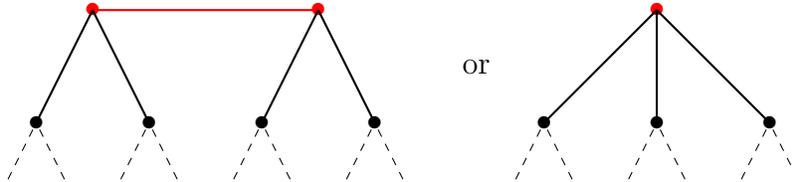
\begin{center}
\begin{tikzpicture}[scale=1.5,sizefont/.style={scale = 2}]

\draw[ultra thick] (1,3) node {${\color{red}\bullet}$};
\draw[ultra thick] (3,3) node {${\color{red}\bullet}$};
\draw[thick,red] (1,3) -- (3,3);

\draw[ultra thick] (0.5,2) node {$\bullet$};
\draw[ultra thick] (1.5,2) node {$\bullet$};
\draw[ultra thick] (2.5,2) node {$\bullet$};
\draw[ultra thick] (3.5,2) node {$\bullet$};
\draw[thick] (1,3) -- (0.5,2);
\draw[thick] (1,3) -- (1.5,2);
\draw[thick] (3,3) -- (2.5,2);
\draw[thick] (3,3) -- (3.5,2);

\draw[dashed] (.25,1.5) -- (0.5,2);
\draw[dashed] (.75,1.5) -- (0.5,2);
\draw[dashed] (2.25,1.5) -- (2.5,2);
\draw[dashed] (2.75,1.5) -- (2.5,2);
\draw[dashed] (1.25,1.5) -- (1.5,2);
\draw[dashed] (1.75,1.5) -- (1.5,2);
\draw[dashed] (3.25,1.5) -- (3.5,2);
\draw[dashed] (3.75,1.5) -- (3.5,2);

\draw[ultra thick] (6,3) node {${\color{red}\bullet}$};
\draw[ultra thick] (5,2) node {$\bullet$};
\draw[ultra thick] (6,2) node {$\bullet$};
\draw[ultra thick] (7,2) node {$\bullet$};

\draw[dashed] (5.25,1.5) -- (5,2);
\draw[dashed] (4.75,1.5) -- (5,2);
\draw[dashed] (6.25,1.5) -- (6,2);
\draw[dashed] (5.75,1.5) -- (6,2);
\draw[dashed] (7.25,1.5) -- (7,2);
\draw[dashed] (6.75,1.5) -- (7,2);

\draw[thick] (6,3) -- (6,2);
\draw[thick] (6,3) -- (7,2);
\draw[thick] (6,3) -- (5,2);

\draw[thick] (4.4,2.5) node {or};
\end{tikzpicture}
\end{center}
respectively (with the crater highlighted in red).  If $\disc \mathcal{O}_0 \equiv 1 \pmod{8}$ then the crater forms a cycle whose length is the order of a certain element in the class group of $\mathcal{O}_0$, as depicted in the following figure.

\begin{center}
\begin{tikzpicture}[scale=1.5,sizefont/.style={scale = 2}]
\draw[ultra thick] (0,1) node {${\color{red}\bullet}$};
\draw[ultra thick] (-1,0) node {${\color{red}\bullet}$};
\draw[ultra thick] (-2/3,-1) node {${\color{red}\bullet}$};
\draw[ultra thick] (2/3,-1) node {${\color{red}\bullet}$};

\draw[ultra thick,red] (0,1) -- (-1,0);
\draw[ultra thick,red, dotted] (0,1) -- (1,0);
\draw[ultra thick,red,dotted] (1,0) -- (2/3,-1);
\draw[ultra thick,red]  (-2/3,-1) -- (2/3,-1);
\draw[ultra thick,red]  (-2/3,-1) -- (-1,0);

\draw[ultra thick] (-1,-4/3) node {${\bullet}$};
\draw[ultra thick] (-4/3,-4/3) node {${\bullet}$};
\draw[ultra thick] (-1,-5/3) node {${\bullet}$};

\draw[ultra thick] (-1,-4/3) -- (-4/3,-4/3);
\draw[ultra thick] (-1,-4/3) -- (-1,-5/3);
\draw[ultra thick] (-1,-4/3) -- (-2/3,-1);

\draw[ultra thick, dotted] (-4/3,-4/3) -- (-5/3,-5/3);
\draw[ultra thick, dotted] (4/3,-4/3) -- (5/3,-5/3);

\draw[ultra thick, dotted] (-4/3,-4/3) -- (-5/3,-3/3);
\draw[ultra thick, dotted] (4/3,-4/3) -- (5/3,-3/3);

\draw[ultra thick, dotted] (-3/3,-5/3) -- (-4/3,-6/3);
\draw[ultra thick, dotted] (3/3,-5/3) -- (4/3,-6/3);

\draw[ultra thick, dotted] (-3/3,-5/3) -- (-2/3,-6/3);
\draw[ultra thick, dotted] (3/3,-5/3) -- (2/3,-6/3);

\draw[ultra thick, dotted] (-6/3,1/3) -- (-8/3,1/3);
\draw[ultra thick, dotted] (6/3,1/3) -- (8/3,1/3);

\draw[ultra thick, dotted] (-6/3,-1/3) -- (-8/3,-1/3);
\draw[ultra thick, dotted] (6/3,-1/3) -- (8/3,-1/3);

\draw[ultra thick, dotted] (-6/3,1/3) -- (-6/3,3/3);
\draw[ultra thick, dotted] (6/3,1/3) -- (6/3,3/3);

\draw[ultra thick, dotted] (-6/3,-1/3) -- (-6/3,-3/3);
\draw[ultra thick, dotted] (6/3,-1/3) -- (6/3,-3/3);

\draw[ultra thick, dotted] (1/3,6/3) -- (1/3,8/3);
\draw[ultra thick, dotted] (1/3,6/3) -- (3/3,6/3);

\draw[ultra thick, dotted] (-1/3,6/3) -- (-1/3,8/3);
\draw[ultra thick, dotted] (-1/3,6/3) -- (-3/3,6/3);

\draw[ultra thick] (1,-4/3) node {${\bullet}$};
\draw[ultra thick] (4/3,-4/3) node {${\bullet}$};
\draw[ultra thick] (1,-5/3) node {${\bullet}$};

\draw[ultra thick] (1,-4/3) -- (4/3,-4/3);
\draw[ultra thick] (1,-4/3) -- (1,-5/3);
\draw[ultra thick] (1,-4/3) -- (2/3,-1);

\draw[ultra thick] (-5/3,0) node {${\bullet}$};
\draw[ultra thick] (-6/3,1/3) node {${\bullet}$};
\draw[ultra thick] (-6/3,-1/3) node {${\bullet}$};

\draw[ultra thick] (-1,0) -- (-5/3,0);
\draw[ultra thick] (-5/3,0) -- (-6/3,1/3);
\draw[ultra thick] (-5/3,0) -- (-6/3,-1/3);

\draw[ultra thick,dotted] (1,0) -- (5/3,0);
\draw[ultra thick,dotted] (5/3,0) -- (6/3,1/3);
\draw[ultra thick,dotted] (5/3,0) -- (6/3,-1/3);

\draw[ultra thick] (0,5/3) node {${\bullet}$};
\draw[ultra thick] (1/3,6/3) node {${\bullet}$};
\draw[ultra thick] (-1/3,6/3) node {${\bullet}$};

\draw[ultra thick] (0,1) -- (0,5/3);
\draw[ultra thick] (0,5/3) -- (1/3,6/3);
\draw[ultra thick] (0,5/3) -- (-1/3,6/3);
\end{tikzpicture}
\end{center}

\begin{rmk} \label{descending_remark}
Observe that when $\disc \mathcal{O}_0 \equiv 5 \pmod{8}$ and $E$ is on the crater, then all 2-isogenies from $E$ are descending. 
\end{rmk}

We now discuss some aspects of the volcano $V_q$ in terms of a matrix representation of Frobenius.  We will use this material in the proof of Theorem \ref{nathan_thm}.  We continue with the notation from earlier in this section.  If $p$ is a prime number, then the Frobenius endomorphism at $p$ has a representative conjugacy class in $\GL_2(\Z_2)$ via the 2-adic representation.  Let $F \in \GL_2(\Z_2)$ be a matrix in this conjugacy class. For any positive integer $k$, we have $\det(F) \equiv q \pmod{2^k}$.  We note that a unit in $\Z_2$ is a square in $\Z_2$ if and only if it is $1$ modulo $8$.  Therefore, it still makes sense to take $\sqf(\alpha) \pmod{8}$ for an $\alpha\in \Z_2$.

Suppose $F = -I + 2^m M$ where $m \ge 2$ and $M =\left( \begin{smallmatrix} x & y \\ z & w\end{smallmatrix} \right)\in \Mat_2(\Z_2)$. This implies 
\[
q \equiv (-1+2^m x)(-1+2^m w)- 2^{2m} yz \equiv 1 -2^m(x+w) - 2^{2m}(yz-xw) \pmod{2^k},
\]
and
\begin{eqnarray*}
t^2 - 4q & \equiv & (-2+2^m (x+ w))^2 - 4 \left((-1+2^m x)(-1+2^m w)- 2^{2m} (yz-xw)\right) \pmod{2^k}\\
& \equiv & 2^{2m} \left((x-w)^2+4yz \right) \pmod{2^k}.
\end{eqnarray*}
Moreover, 
\[
\sqf(t^2-4q) \equiv \sqf\left((x-w)^2 + 4yz\right) \pmod{8}.
\]
Therefore, we have the following:
\begin{enumerate}
\item $v_2(\disc \mathcal{O}_0)$ is determined by $\sqf\left((x-w)^2 + 4yz\right) \pmod{8}$, and 
\item $h(V_q)$ is determined by 
\begin{itemize}
\item $v_2\left((x-w)^2+4yz\right)$, and 
\item $\sqf\left((x-w)^2 + 4yz\right) \pmod{8}$.
\end{itemize}
\end{enumerate}

\section{Elliptic Curves over $\Q$} \label{Q}

We now turn to the proof of Theorem \ref{mainthm2}.  Let $E, E'$ be rationally 2-isogenous elliptic curves defined over $\Q$.  Because the 2-isogeny is defined over $\Q$, each curve has at least a rational 2-torsion point.  The exact proportion of anomalous primes is determined by the images of the 2-adic representations of $E$ and $E'$, as we will see below.  For the remainder of this section we will assume that both $G \ddef \im \rho_{E,2}$ and $G' \ddef \im \rho_{E',2}$ have index 3 in $\GL_2(\Z_2)$.  Up to isomorphism, $\GL_2(\Z_2)$ has a unique subgroup of index 3. 

\subsection{Frobenius at Anomalous Primes} \label{rep_setup} 
In this  section we will describe the conjugacy class in $\GL_2(\Z_2)$ associated to Frobenius at an anomalous prime $p$.

If $p$ is anomalous then both $E$ and $E'$ have $E(\F_p)[2^\infty] \simeq E'(\F_p)[2^\infty] \simeq \Z/2\Z \times \Z/2 \Z$ by Lemma \ref{22lem}.  Write $F$ and $F'$ for matrix representatives of the Frobenius classes of $E$ and $E'$, respectively, as elements of $\GL_2(\Z_2)$.  It follows that 
\[
F \equiv F' \equiv I \pmod{2}
\]
and that neither $F \pmod{4}$ nor $F' \pmod{4}$ fixes a cyclic subgroup of $\Z/4\Z \times \Z/4\Z$ of order $4$. 

Since anomalous primes can be partitioned by defect as in Theorem \ref{establish_defect}, let us fix $m \geq 2$ and suppose that $p$ has defect $(m+1,m)$.  In particular, we assume that the isogeny $E \to E'$ is descending. Then we have
\[
E(\F_{p^2})[2^{\infty}] \simeq \Z/2^{a}\Z \times \Z/2^{m+1}\Z \text{ \qquad and \qquad  } E'(\F_{p^2})[2^{\infty}] \simeq \Z/2^{a+1}\Z \times \Z/2^{m}\Z,
\]
where $a \geq m+1$.  Therefore
\begin{align*}
F^2 &\equiv I \pmod{2^{m+1}} \text{ but } F^2\not\equiv I \pmod{2^{m+2}},\text{ and} \\
(F')^2 &\equiv I \pmod{2^{m}} \text{ but } (F')^2\not\equiv I \pmod{2^{m+1}}. 
\end{align*}

We are thus led to the problem of determining, for fixed $m \geq 2$, matrices $A \in \GL_2(\Z_2)$ such that the following are simultaneously satisfied 
\begin{itemize}
\item $A \equiv I \pmod{2}$, and
\item $A \pmod{4}$ does not fix any cyclic subgroup of $\Z/4\Z \times \Z/4\Z$ of order 4, and
\item $A^2 \equiv I \pmod{2^{m+1}}$ \text{ but } $A^2 \not\equiv I \pmod{2^{m+2}}$.
\end{itemize}
It is now an exercise in squaring matrices (which we omit) to conclude that there exist matrices $M,M' \in \Mat_2(\Z_2)$ such that neither $M$ nor $M'$ is $\equiv 0 \pmod{2}$ and that $F$ and $F'$ are, up to conjugation, given by
\begin{align*}
F &= -I + 2^mM \\
F' &= -I + 2^{m-1}M'.
\end{align*}

We finish this subsection by collecting some known results on the Galois theory of torsion point fields and their consequences for anomalous primes.  The important point is that if $k$ is a number field and $E/k$ is an elliptic curve for which $k(E[\ell^n])/k$ has Galois group $\GL_2(\Z/\ell^n\Z)$, then the normal subgroup $\lbrace \pm I \rbrace$ of $\GL_2(\Z/\ell^n\Z)$ is the Galois group of $k(E[\ell^n])/k(x(E[\ell^n]))$, with clear implications for the Frobenius at anomalous primes.

\begin{prop} \label{adelman_prop}
Let $k$ be a number field and $E/k$ an elliptic curve.  Let $\ell$ be a prime number and $n \geq 1$ an integer.  Let $k(E[\ell^n])$ be the $\ell^n$-torsion field of $E$ and $k(x(E[\ell^n]))$ the subfield generated by the $x$-coordinates of the points of $E[\ell^n]$.  Let $G(\ell^n) = \im \overline{\rho}_{{E,\ell^n}} \subseteq \GL_2(\Z/\ell^n\Z)$ be the image of the mod $\ell^n$ representation.  Then  $[k(E[\ell^n]):k(x(E[\ell^n]))] \leq 2$ with $\Gal(k(E[\ell^n])/k(x(E[\ell^n]))) \simeq G(\ell^n) \cap \lbrace \pm I \rbrace$.
\end{prop} 

\begin{proof}
This is contained in \cite[Ch.~5]{adelman}; see especially Figs.~5.4, 5.5, 5.7.
\end{proof}

\begin{lem} \label{adelman_lem}
Let $E$ be an elliptic curve over $\Q$ and suppose $p\ne 2$ is a good prime for $E$.  Let $K_{2^n} = \Q(E[2^n])$ with Galois group $\Gal(K_{2^n}/\Q) \simeq G(2^n) \subseteq \GL_2(\Z/2^n\Z)$.  Suppose $\Frob_p \in \Gal(K_{2^n}/\Q)$ is a lift of the Frobenius automorphism at $p$ (so that the decomposition group of $K_{2^n}$ is generated by $\Frob_p$) and suppose that $\overline{\rho}_{E,2^n}(\Frob_p) = F = - I \in G(2^n)$. Then $\F_p(x(E[2^n])) = \F_p$ and $\F_p$ contains no $y$-coordinate of any $2^n$-torsion point of $E$.
\end{lem}

\begin{proof}
This is a matter of translating the arithmetic of elliptic curves into the Galois theory of torsion point fields and the  behavior of Frobenius at unramified primes. In particular, it is the ``reduction modulo $p$'' of Proposition \ref{adelman_prop}.  

Since $p$ is an odd prime of good reduction for $E$ it is unramified in $K_{2^n}$, hence we can appeal to the explicit polynomial descriptions in \cite[Table 5.1]{adelman}.  Let $K_{2^n}$ be the splitting field of the polynomial $T_{2^n}(x)$ and $\Q(x(E[2^n]))$ the splitting field of $\Lambda_{2^n}(x)$. 

In general, the field extension $K_{2^n}/\Q(x(E[2^n]))$ has degree 1 or 2, depending on whether $\Q(x(E[2^n]))$ contains any $y$-coordinates of any $2^n$-torsion points (note that if $G(2^n) = \GL_2(\Z/2^n\Z)$ then the extension has degree 2).  We have $\Gal(K_{2^n}/\Q(x(E[2^n]))) \simeq \lbrace \pm I \rbrace \cap G(2^n)$ by Proposition \ref{adelman_prop}, $\Lambda_{2^n}(x)$ splits completely in $\Q(x(E[2^n]))$, and that $K_{2^n}$ is generated over $\Q(x(E[2^n]))$ by a single $y$-coordinate of a single $2^n$-torsion point (see \cite[p.~74]{adelman}).  

The Galois theory of number fields then says that either $T_{2^n}(x)$ splits completely over $\Q(x(E[2^n]))$ or factors as a product of irreducible quadratic polynomials, each of them Galois-conjugate.  In either case, the Galois group $\Gal(K_{2^n}/\Q(x(E[2^n])))$ is the decomposition group at $p$, which is isomorphic to $\langle \Frob_p \rangle$.  The hypothesis that $F \equiv -I \pmod{2^n}$ means that the polynomial $\Lambda_{2^n}(x)$ splits completely modulo $p$, hence $\F_p(x(E[2^n])) = \F_p$.  The fact that Frobenius is non-trivial implies that $\F_{p}(E[2^n])$ is a quadratic extension of $\F_p$, hence contains no $y$-coordinate of any $2^n$-torsion point of $E$.
\end{proof}

Next, we recall a basic fact about towers of torsion fields.  

\begin{thm} \label{lattes_tower}
Let $k$ be a field, $\ell$ a prime number, and $E/k$ an elliptic curve. Then we have the following inclusions of fields for all $n \geq 1$:
\[
\xymatrix{
& \\
  & k(E[\ell^n]) \ar@{-}[d] \ar@{--}[u] \\
\ar@{-}[ur] k(x(E[\ell^n])) \ar@{-}[d] \ar@{--}[u]& k(E[\ell^{n-1}]) \ar@{-}[d] \\
\ar@{-}[ur] k(x(E[\ell^{n-1}])) \ar@{--}[d] & k(E[\ell^{n-2}]) \ar@{--}[d] \\
\ar@{-}[ur] & }
\]
\end{thm}

\begin{cor} 
With all notation as above, suppose $E/\F_p$ is an elliptic curve such that $\F_p(x(E[2^n])) = \F_p$.  Then $\F_p(x(E[2^k])) = \F_p$ for all $k \leq n$. 
\end{cor}

\begin{proof}
This follows immediately.
\end{proof}

\begin{rmk}
One can see this from a representation theory point of view too: if $F \equiv -I \pmod{2^n}$, then $F \equiv -I \pmod{2^k}$ for all $k\leq n$ as well.
\end{rmk}

The next proposition shows that over a finite field $\F_p$, if $E \to E'$ is descending and $F \equiv -I \pmod{2^m}$ then we automatically get that $F' \equiv -I \pmod{2^{m-1}}$.  This does not immediately imply that that $p$ is anomalous because it could further be the case that $F' \equiv -I \pmod{2^{m}}$.  This will be used in the proof of Theorem \ref{m_proportion} below where we argue that $F' \equiv -I \pmod{2^m}$ for half of the primes for which $F \equiv -I \pmod{2^m}$  and $F' \equiv -I \pmod{2^{m-1}}$ for the other half.

\begin{prop} \label{F'form}
Let $E$ and $E'$ be ordinary 2-isogenous elliptic curves defined over $\F_p$ and suppose that the isogeny $E \to E'$ is descending.  Suppose $E(\F_p)[2^\infty] \simeq E'(\F_p)[2^\infty] \simeq \Z/2\Z  \times \Z/2\Z$ and that $F \equiv -I \pmod{2^m}$.  Then $F' \equiv -I \pmod{2^{m-1}}$.
\end{prop}

\begin{proof}
Since $F \equiv -I \pmod{2^m}$, we have $E(\F_p)[2^\infty] \simeq \Z/2\Z \times \Z/2\Z$ and $E(\F_{p^2})[2^{m+1}] \simeq \Z/2^{m+1}\Z \times \Z/2^{m+1}\Z$. Since $E$ and $E'$ are isogenous, the groups $E(\F_{p^2})$ and $E'(\F_{p^2})$ have the same size, hence their 2-Sylow subgroups have the same size.  

If the 2-Sylow subgroups over $\F_{p^2}$ are isomorphic, then $(F')^2 \equiv I \pmod{2^{m+1}}$. It is also the case that $F' \equiv I \pmod{2}$ and $F$ does not fix a cyclic subgroup of $\Z/4\Z \times \Z/4\Z$ of order 4.  A calculation with matrices shows that $F' = -I \in \GL_2(\Z/2^m\Z)$ is the unique matrix satisfying these conditions simultaneously.  Thus, $F' \equiv -I \pmod{2^m}$. Hence it is also true that $F' \equiv -I \pmod{2^{m-1}}$.  

If the 2-Sylow subgroups of $\F_{p^2}$ are not isomorphic, then because the isogeny is descending we have $E(\F_{p^2})[2^{m}] \simeq \Z/2^{m}\Z \times \Z/2^{m}\Z$. Hence $F'$ is a matrix such that $F' \equiv I \pmod{2}$, does not stabilize a cyclic subgroup of $\Z/4 \Z \times \Z/4\Z$ of order 4, and satisfies $(F')^2 \equiv I \pmod{2^m}$.  Therefore $F' \equiv -I \pmod{2^{m-1}}$ by the same reasoning.
\end{proof}

\begin{rmk}
This proposition tells us that if $\F_p(x(E[2^n])) = \F_p$ then $\F_p(x(E'[2^{n-1}])) = \F_p$.
\end{rmk}

To finish off this section we will record a technical lemma that we will need in the proof of Theorem \ref{m_proportion} below.

\begin{lem} \label{splitting_lem}
Let $E$ be an elliptic curve over a field $k$ of characteristic $p>3$ and write 
\[
E:~y^2 = x^3 + ax+b.
\]  
Suppose that $k$ contains $x(E[2^n])$.  Let $P = (\xi,\eta)$ be a point of order $2^{n+1}$ and let $\langle P \rangle$ denote the cyclic subgroup of $E[2^{n+1}]$ generated by $P$.  Then the set of $x$-coordinates of the points in $\langle P \rangle$ are contained in $k$ if and only if the the polynomial
\[
x^4 - 4\xi x^3 - 2ax^2 + (-4\xi a - 8b)x + (a^2 - 4\xi b)
\]
splits in $k$.
\end{lem}

\begin{proof}
The difference between any two points in $\langle P \rangle$ is a point of order dividing $2^{n}$.  By Theorem \ref{lattes_tower}, since $k$ contains $x(E[2^n])$, it contains the $x$-coordinates of all points of order dividing $2^n$.  Thus, one point of $\langle P \rangle$ of exact order $2^{n+1}$ will have rational $x$-coordinate if and only if they all do.  Therefore, all the points of $\langle P \rangle$ will have rational $x$-coordinates if and only if the points of exact order $2^{n+1}$ do.  Such a point $P$ is the preimage under the duplication map of a point of order $2^n$ in $\langle P \rangle$, hence by \cite[III.2.3(d)]{silverman} the $x$-coordinate of $P$ is $k$-rational if and only if the quartic polynomial (whose roots are the $x$-coordinates of these points of order $2^{m+1}$)
\[
x^4 - 4\xi x^3 - 2ax^2 + (-4\xi a - 8b)x + (a^2 - 4\xi b)
\]
has all its roots defined over $k$.  
\end{proof}

\subsection{The Proportion of Anomalous Primes} 

Fix $m \ge 2$.  The key step in proving Theorem \ref{mainthm2} is the following.

\begin{thm} \label{m_proportion}
Suppose $E$ and $E'$ are rationally 2-isogenous elliptic curves defined over $\Q$ such that $G$ and $G'$ each have index 3 in $\GL_2(\Z_2)$.  Let $m \geq 2$.  Then the proportion of anomalous primes of defect $(m+1,m)$ is 
\[
\frac{1}{2} \cdot \frac{1}{|G(2^m)|} = \frac{1}{2^{4m+2}}.
\]
\end{thm}

\noindent We break this proof into two steps, starting with a Lemma.

\begin{lem} \label{lemA}
Suppose $p$ is a prime for which $F \equiv  -I + 2^mM \pmod{2^{m+1}}$ with $M \not \equiv 0 \pmod{2}$.  Then $p$ is anomalous of defect $(m+1,m)$ if and only if $\F_p(x(E'[2^m])) \ne  \F_p$.
\end{lem}

\begin{proof}
This follows from  the results of the previous section. We have that $p$ is anomalous of defect $(m+1,m)$ if and only if $E(\F_p)[2^\infty] \simeq E'(\F_p)[2^{\infty}] \simeq \Z/2\Z \times \Z/2\Z$, $E(\F_{p^2})[2^\infty] \simeq \Z/2^a\Z \times \Z/2^{m+1}\Z$, and $E'(\F_{p^2}) \simeq \Z/2^{a+1}\Z \times \Z/2^{m}\Z$.  By our matrix calculations, this is true if and only if $F \equiv -I + 2^mM \pmod{2^{m+1}}$ and $F' \equiv -I + 2^{m-1}M'\pmod{2^{m}}$ with neither $M$ nor $M' \equiv 0 \pmod{2}$.  By Proposition \ref{F'form}, $F \equiv -I + 2^mM\pmod{2^{m+1}}$ implies $F' \equiv -I + 2^{m-1}M'\pmod{2^{m}}$.  By the Galois theory of torsion point fields from Lemma \ref{adelman_lem}, $M' \not \equiv 0 \pmod{2}$ if and only if $\F_p(x(E'[2^m])) \ne  \F_p$.
\end{proof}

We now make some global choices for $E$ that we use in the next proof.  Fix a basis $P,Q$ for $T_2E$ and write $P_{2^k}$, $Q_{2^k}$ for the reductions modulo $2^k$ of $P$ and $Q$, respectively.  Since $E$ and $E'$ are rationally 2-isogenous, there exists a 2-isogeny $\varphi: E \to E'$ defined over $\Q$.  Since $\varphi$ is defined over $\Q$, we may write $E' = E/\langle S \rangle$ for some $\Q$-rational 2-torsion point $S$ of $E$.  We choose our basis so that $S = P_2$. 

Let $P' = \varphi(P)$ and $Q' = \varphi(Q)$ with $P'_{2^k}$ and $Q'_{2^k}$ defined similarly.  We have that $Q'_{2^k}  = Q_{2^k} + \langle P_2 \rangle$ is a $2^k$-torsion point of $E'$, but $P'_{2^k}$ is not necessarily independent of $Q'_{2^k}$. We fix a basis $Q',R'$ for $T_2E'$ so that for all $m,\ Q'_{2^m}$ and $R'_{2^m}$ form a basis for $E'[2^m]$. 

By applying V\'elu's explicit formulas, we see that there exists a change of coordinates such that $E$ and $E'$ are given by the explicit Weierstrass equations
\begin{align*}
E:~y^2 &= (x + a_2)(x^2 - 4a_4) \\
E':~y^2 &= x(x^2+a_2x + a_4),
\end{align*}
where $P_2 = (-a_2,0)$ and $P_2' = (0,0)$.  Write $R'_{2^{m-1}} = (\xi_{m-1},\eta_{m-1})$ with $\xi_{m-1},\eta_{m-1} \in \overline{\Q}$.  Then the $x$-coordinate $\xi_m$ of $R'_{2^m}$ is given by one of the roots the quartic 
\begin{align} \label{quartic}
x^4 - 4\xi_{m-1} x^3 + (-4\xi_{m-1} a_2 + 6a_4)x^2 + (4a_4a_2 - 8a_4)x + (-4a_4a_2 + (a_4^2 - 4\xi_{m-1} a_4)).
\end{align}

Next we show the existence of anomalous primes of defect $(m+1,m)$ for all $m \geq 2$.

\begin{thm} \label{nathan_thm}
Let $E$ and $E'$ be rationally 2-isogenous elliptic curves over $\Q$ and suppose that $G$ and $G'$ each have index 3 in $\GL_{2}(\Z_2)$.  Then for all $m \geq 2$ there exist anomalous primes of defect $(m+1,m)$. 
\end{thm}

\begin{proof}
Fix $m \geq 2$. By the assumption on the size of $G$ and $G'$ and the Chebatorev Density Theorem, there exist infinitely many primes p for which $F \pmod{2^{m+1}}$  is in the conjugacy class of 
\[
-I + 2^m \begin{pmatrix} 1 & 1 \\ 1 & 0 \end{pmatrix}.
\]
Let $V_p$ be the isogeny volcano which contains $E$ and $E'$ as adjacent vertices and let $V_{p^2}$ be the corresponding volcano over $\F_{p^2}$.  By our work in Section \ref{volcanoes} we know that $E$ is on level at least $m$ of $V_p$.

\medskip

\noindent \textbf{Claim.} 
The height of $V_p$ is $m$ and $\disc \mathcal{O}_0 \equiv 5 \pmod{8}$.

\medskip

We now prove the claim.
Write $F = -I + 2^m \left(\begin{smallmatrix} x&y \\ z & w \end{smallmatrix} \right) \in \GL_2(\Z_2)$.  
As in the end of Section~\ref{volcanoes}, 
\begin{align*}
t^2 - 4p &\equiv  2^{2m}((x-w)^2+4yz) \pmod{2^{2m+1}}.
\end{align*}
Therefore $v_2(t^2-4p) = 2m + v_2((x-w)^2+4yz)$.

We have $\sqf(t^2-4p) \equiv \sqf( (x-w)^2+4yz)) \pmod{8}$.  If $x-w$ is odd, then $v_2((x-w)^2+4yz) = 0$ and $v_2(t^2-4p) = 2m$.  In this case it is also true that $\sqf( (x-w)^2+4yz)) \equiv 5 \pmod{8}$ if and only if $yz$ is odd.  If this holds, $\disc \mathcal{O}_0 \equiv 5 \pmod{8}$.  Consequently, since $\disc \mathcal{O}_0 \equiv 1 \pmod{4}$, equation \eqref{eqn:volcano_height} shows that the height of $V_p$ is $v_2(t^2-4p)/2 = m$.

Now set $\left(\begin{smallmatrix} x&y \\z & w \end{smallmatrix} \right) = \left(\begin{smallmatrix} 1&1 \\ 1 & 0 \end{smallmatrix} \right)$ as above.  We saw in Section \ref{volcanoes} that for a volcano $V_p$ in which $\disc \mathcal{O}_0 \equiv 5 \pmod{8}$, there is a unique vertex on the crater.  Since $E$ lies on level at least $m$ of $V_p$ and $h(V_p)= m$, we see that $E$ is the unique vertex on the crater of $V_p$.  

\medskip

Returning to the proof of the theorem, we conclude from the Claim that $V_{p^2}$ has height $m+1$ and the group structure on the crater is $\Z/2^a\Z \times \Z/2^{m+1}\Z$.  Since a vertex on the floor of $V_{p^2}$ has a cyclic group of rational points, it must be the case that the curves on each level of $V_{p^2}$ have different group structures.
So in particular, $E'(\F_{p^2})[2^\infty] = \Z/2^{a-1}\Z \times \Z/2^m\Z$.
This means that $p$ must have defect $(m+1,m)$.
\end{proof}

We now finish the proof of Theorem \ref{m_proportion}.

\begin{proof}[Proof of Theorem \ref{m_proportion}]
By Lemma \ref{lemA} a prime $p$ is anomalous of defect $(m+1,m)$ if and only if $F = -I + 2^mM$ and $F' = -I + 2^{m-1}M'$, with neither $M$ nor $M' \equiv 0 \pmod{2}$.   We will interpret our proportion $1/{2^{4m+2}}$ as a conditional probability.  Suppose $p$ is a prime such that $F = -I + 2^mM$.  By the Chebotarev Density Theorem, the proportion of such primes is $1/|G(2^m)| = 1/2^{4m+1}$.  By Proposition \ref{F'form}, we have that $F' \equiv -I \pmod{2^{m-1}}$ at these primes as well.  We will show that for a proportion of 1/2 of these primes we have $F' \not\equiv -I \pmod{2^m}$ and so $p$ is anomalous with defect $(m+1,m)$.  Now we compute in the basis we have set above:
\begin{align*}
F(P_{2^m}) &= -P_{2^m} \\
F(Q_{2^m}) &= -Q_{2^m}
\end{align*}
so that
\[
F'(Q'_{2^m}) = F'(Q_{2^m} + \langle P_2 \rangle) = F(Q_{2^m}) + F(\langle P_2 \rangle) =  -Q_{2^m} + \langle P_2 \rangle = -Q'_{2^m} 
\]
because $P_2$ is defined over $\Q$.   

Therefore, we have determined that $F'$ acts on $E'[2^m]$ via 
\[
F' \equiv \begin{pmatrix}
-1 & * \\ 0 & *
\end{pmatrix} \pmod{2^m}.
\]
But since $p \equiv \det(F) \pmod{2^m}$ and $\det(F) \equiv 1 \pmod{2^m}$, we must additionally have $F' \equiv \left(\begin{smallmatrix} -1 & * \\ 0 & -1 \end{smallmatrix} \right) \pmod{2^m}$.  Therefore, Proposition \ref{adelman_prop} shows that in this setting $p$ is not anomalous of defect $(m+1,m)$ if and only all the $x$-coordinates of the $2^m$-torsion points of $E'$ are defined over $\F_p$.  To determine when this happens, we examine the quartic (\ref{quartic}).

By Propositions \ref{adelman_prop} and \ref{F'form}, we know that the the $x$-coordinates of the $2^{m-1}$-torsion points on $E'$ are $\F_p$-rational.  Any two choices of $R'_{2^m}$ differ by a $2^{m-1}$-torsion point. Therefore, the $x$-coordinate of any choice of $R'_{2^m}$ is defined over $\F_p$ if and only if  the $x$-coordinate of one choice of of $R'_{2^m}$ is defined over $\F_p$.

With notation as above, consider the quartic polynomial given in \eqref{quartic}.  The roots of this polynomial give the $x$-coordinates of the $2^m$-torsion points of all preimages of $R_{2^{m-1}}$.

Since the $x$-coordinates of all of the $2^m$-torsion points of $E$ are defined over $\F_{p^2}$, the quartic polynomial in \eqref{quartic} must factor over $\F_p$ as a product of irreducible polynomials each of degree at most $2$.  In particular, it is reducible over $\F_p$.

We have shown that if the quartic (\ref{quartic}) has one root defined over $\F_p$, then it splits completely into linear factors over $\F_p$.  Therefore, since (\ref{quartic}) is reducible over $\F_p$, it factors as a product of two conjugate quadratic polynomials over $\F_p$.  If it were the case that these polynomials split into linear factors over $\F_p$ for every $p$, there would not exist any primes of defect $(m+1,m)$, contradicting Theorem \ref{nathan_thm}. Thus they must be irreducible for 1/2 of the primes considered in this proof and and split for the complementary primes, and so the proportion of primes of defect $(m+1,m)$ is $(1/2) \cdot (1/2^{4m+1}) = 1/2^{4m+2}$, as claimed.
\end{proof}

We now complete the proof of Theorem \ref{mainthm2} as a corollary.

\begin{cor}
With all notation as above, we have $\mathcal{P} = 1/30$.
\end{cor}

\begin{proof}
For all $m \geq 2$, Theorem \ref{m_proportion} shows that the proportion of anomalous primes of defect $(m+1,m)$ is $2^{-4m-2}$.  By symmetry via the dual isogeny, the proportion of anomalous primes of defect $(m,m+1)$ is $2^{-4m-2}$ as well.  Therefore, the proportion of anomalous primes $\mathcal{P}$ is given by the geometric series
\[
\mathcal{P} = 2\sum_{m =2}^{\infty} \frac{1}{2^{4m + 2}} = \frac{1}{32} \sum_{k=0}^{\infty} \frac{1}{16^k} = \frac{1}{30}.
\]
\end{proof}

\section{The Distribution of Anomalous Primes by Volcano Height}

In this section we take a different point of view and explore how the defect of an anomalous prime corresponds to the height and shape of the associated volcano.  These results are motivated by experiments with the pair $(E,E'$) of rationally 2-isogenous elliptic curves over $\Q$ where where $E$ has \texttt{LMFDB} label \href{https://www.lmfdb.org/EllipticCurve/Q/69/a/2}{{\tt 69a2}}  and $E'$ has label \href{https://www.lmfdb.org/EllipticCurve/Q/69/a/1}{{\tt 69a1}}.  We computed the anomalous primes $p$ up to $2\cdot 10^7$ and divided them up by defect, the height of the associated volcano $h(V_p)$, and $\disc \mathcal{O}_0 \pmod{8}$, which determines the shape of the crater of $V_p$.  We include the data for anomalous primes of defect $(3,2)$ and for anomalous primes of defect $(4,3)$ in Appendix \ref{calculations}.  

Let $S_m$ be the set of anomalous primes of defect $(m+1,m)$.  For $i \in \{0,1,4,5\}$ and a positive integer $H \ge m$, let $S_m(i,H)$ be the subset of $p\in S_m$ for which $\disc \mathcal{O}_0 \equiv i \pmod{8}$ and $h(V_p) = H$.  Let $S'_m(i,H)$ denote the proportion of primes in $S_m$ that lie in $S_m(i,H)$.  The data we have collected strongly suggest the following results.  
\begin{conj}\label{volcano_conj}
Let $E$ and $E'$ be rationally 2-isogenous elliptic curves over $\Q$ such that $[\GL_2(\Z_2): \im \rho_{E,2}] = [\GL_2(\Z_2):\im \rho_{E',2}] = 3$.  For any $H \ge m$, we have
\[
S_m'(1,H) = S_m'(5,H) = 4^{-(H-(m-1))}
\]
and 
\[
S_m'(0,H) = S_m'(4,H) =  \frac{1}{2}\cdot 4^{-(H-(m-1))}.
\]
\end{conj}
We give one quick check that this conjecture is reasonable.  Since every $p \in S_m$ lies in exactly one of the sets $S_m(i,H)$, it must be the case that 
\[
\sum_{i\in \{0,1,4,5\}} \sum_{H \ge m} S_m(i,H) = 1.
\]  
For $i \in \{1,5\}$ we have 
\[
\sum_{H \ge m} S_m(i,H) = \sum_{H \ge m} 4^{-(H-(m-1))} = \frac{1}{4} \cdot \frac{1}{1-\frac{1}{4}} = \frac{1}{3}.
\]
For $i \in \{0,4\}$ we have 
\[
\sum_{H \ge m} S_m(i,H) = \sum_{H \ge m} \frac{1}{2} \cdot 4^{-(H-(m-1))} = \frac{1}{8} \cdot \frac{1}{1-\frac{1}{4}} = \frac{1}{6}.
\]

There are three possibilities for the shape of the crater of the volcano $V_p$, depending on whether $\disc \mathcal{O}_0$ is congruent to $1$ modulo $8,\ 5$ modulo $8$, or $0$ modulo $4$.  This calculation suggests that among the set of all anomalous primes, these three shapes are equally likely, and further that if we divide up the volcanoes of a fixed height $H \ge m$, all three crater shapes are equally likely.  Another nice consequence of this conjecture is that for any fixed $i \in \{0,1,4,5\}$ it is clear how $S_m(i,H)$ changes with $H$, as it predicts that 
\[
\frac{S_m(i,H+1)}{S_m(i,H)} = \frac{1}{4}.
\]

In Section \ref{Q} we saw that if $p$ is anomalous of defect $(m+1,m)$ and $F \in \GL_2(\Z_2)$ is in the conjugacy class of Frobenius, then 
\[
F = -I + 2^m \left(\begin{array}{cc} x & y \\ z & w
\end{array}\right)
\]
where $x,y,z,w$ are not all $0\pmod{2}$.  At the end of Section \ref{volcanoes} we saw that $\disc \mathcal{O}_0 \pmod{8}$ is determined by $\sqf((x-w)^2 + 4yz) \pmod{8}$ and that $h(V_p)$ is determined by both $\disc \mathcal{O}_0 \pmod{8}$ and $v_2\left((x-w)^2+4yz\right)$.

The goal of this section is to show that if the matrix $\left(\begin{smallmatrix} x & y \\ z & w \end{smallmatrix}\right)$ were distributed like a Haar random matrix in $\Mat_2(\Z_2)$ subject to the additional constraint that $v_2(y) = 0$, we would see the behavior predicted in Conjecture \ref{volcano_conj}.  We do not currently have a satisfactory explanation of why Frobenius at anomalous primes of defect $(m+1,m)$ should correspond to these `random matrices with $y$ odd'.  

Fix a positive integer $m \ge 2$.  We now explain our model for anomalous primes of defect $(m+1,m)$.  Let $E$ be an elliptic curve over $\F_p$ with trace of Frobenius $t$ and $V_p$ be the associated 2-isogeny volcano over $\F_p$. Let $K = \Q(\sqrt{t^2-4p}) = \Q(\sqrt{D})$ where $D = \sqf(t^2-4p)$.  Recall from Section \ref{volcanoes} that $h(V_p) = H$ if and only if
\[
v_2(t^2 -4p) = 2m + v_2((x-w)^2 + 4yz) =  \begin{cases} 2H & \text{if }  D\equiv 1 \pmod{4} \\
2H+2 & \text{if } D\equiv 3 \pmod{4} \\
2H+3 & \text{if } D\equiv 2 \pmod{4}. 
\end{cases}
\]
Also recall that $\disc \mathcal{O}_0\equiv \disc \mathcal{O}_K\pmod{8}$.

Instead of starting from an elliptic curve over $\F_p$ we consider a Haar random matrix $M = \left(\begin{smallmatrix} x & y \\ z & w \end{smallmatrix}\right)$ with entries in $\Z_2$ subject to the additional constraint that $v_2(y) = 0$.  We use $\det(-I + 2^m M)$ in place of $p$ and $\operatorname{trace}\left(-I+2^m M\right) = -2+2^m(x+w)$ in place of $t$.  Note that for any fixed $x$, the map taking $w$ to $\alpha = x-w$ is a bijection on $\Z_2$.  For the rest of the proof we usually do not refer to $x$ and $w$, but only to $\alpha$.  Let $\alpha$ and $z$ be random elements of $\Z_2$ distributed with respect to Haar measure, and $y \in \Z_2^*$ be a random unit in $\Z_2$.  We write $\Prob(\cdot)$ to denote the proportion of $\alpha,y,z$ for which some property holds.

We define a kind of height associated to the matrix $M$.  Let 
\[
H_M = \begin{cases}
m + \frac{v_2((x-w)^2 + 4yz)}{2} & \text{if } \sqf((x-w)^2 + 4yz) \equiv 1 \pmod{4} \\
m -1 + \frac{v_2((x-w)^2 + 4yz)}{2} & \text{if } \sqf((x-w)^2 + 4yz) \equiv 3 \pmod{4} \\
m -1 + \frac{v_2((x-w)^2 + 4yz)-1}{2} & \text{if } \sqf((x-w)^2 + 4yz) \equiv 2 \pmod{4}
\end{cases}.
\]

\begin{thm} \label{height_dist}
Let $m \ge 2$ and $H \ge m$ be positive integers.  Let $M = \left(\begin{smallmatrix} x & y \\ z & w \end{smallmatrix}\right) \in \Mat_2(\Z_2)$ be a Haar random matrix subject to the additional constraint that $v_2(y) = 0$.  
\begin{enumerate}
\item For $i \in \{1,5\}$, the probability that $\sqf((x-w)^2 + 4yz) \equiv i \pmod{8}$ and $H_M = H$ is $4^{-(H-(m-1))}$.  
\item For $i \in \{2,3\}$, the probability that $\sqf((x-w)^2 + 4yz) \equiv i \pmod{4}$ and $H_M = H$ is $\frac{1}{2}\cdot 4^{-(H-(m-1))}$.
\end{enumerate}
\end{thm}

Theorem \ref{height_dist} follows from the following stronger result.
\begin{thm}\label{valuation_thm}
\begin{enumerate}
\item 
\begin{eqnarray*}
\Prob\left(\begin{array}{ccc} v_2(\alpha^2+4yz) & = &k \\ \sqf(\alpha^2+4yz) & \equiv & 1\pmod{8} \end{array}\right) & = &  
\Prob\left(\begin{array}{ccc} v_2(\alpha^2+4yz) & = &k \\  \sqf(\alpha^2+4yz) & \equiv & 5\pmod{8} \end{array}\right) \\
& = & 
\begin{cases} 
0 & \text{if } k \text{ is odd},\\
2^{-(k+2)} & \text{if } k \text{ is even}.
\end{cases}
\end{eqnarray*}
\item
\[
\Prob\left(\begin{array}{ccc} v_2(\alpha^2+4yz) & = &k \\  \sqf(\alpha^2+4yz) & \equiv & 3\pmod{4} \end{array}\right) = 
\begin{cases} 
0 & \text{if } k \text{ is odd or } k=0,\\
2^{-(k+1)} & \text{if } k\ge 2 \text{ is even}.
\end{cases}
\]

\item
\[
\Prob\left(\begin{array}{ccc} v_2(\alpha^2+4yz) & = &k  \\ \sqf(\alpha^2+4yz) & \equiv & 2\pmod{4} \end{array}\right) = 
\begin{cases} 
0 & \text{if } k \text{ is even or } k=1,\\
2^{-k} &  \text{if } k\ge 3 \text{ is odd}.
\end{cases}
\]
\end{enumerate}
\end{thm}
It is straightforward to check that this result implies Theorem \ref{height_dist} by checking the base case $H = m$ and then thinking about what happens to these probabilities as we increase $H$, dividing things into cases based on $\sqf(\alpha^2+4yz)\pmod{8}$ and using the definition of $H_M$.

We prove this result by dividing the set of all $\alpha, z \in \Z_2$ and $y \in \Z_2^*$ based on the relative sizes of $v_2(\alpha^2)$ and $v_2(4yz) = 2 + v_2(z)$.  More precisely, we prove Theorem \ref{valuation_thm} in three parts, where each part is divided into cases based on $\sqf(\alpha^2+4yz)\pmod{8}$.

\begin{lem}\label{alpha small}
\begin{enumerate}
\item
\[
\Prob\left(\begin{array}{ccc} v_2(\alpha^2+4yz) & = & k \\  \sqf(\alpha^2+4yz) & \equiv & 1\pmod{8} \\  v_2(\alpha^2)< v_2(4yz)  & & \end{array}  \right)
=
\begin{cases} 
0 & \text{if } k \text{ is odd},\\
2^{-(3k/2+2)} &  \text{otherwise}.
\end{cases}
\]
\item
\[
\Prob\left(\begin{array}{ccc} v_2(\alpha^2+4yz) & = & k \\  \sqf(\alpha^2+4yz) & \equiv & 5\pmod{8} \\  v_2(\alpha^2)< v_2(4yz) & & \end{array} \right)
=
\begin{cases} 
0 & \text{if } k \text{ is odd},\\
2^{-(3k/2+2)} &  \text{otherwise}.
\end{cases}
\]

\item
\[
\Prob\left(\begin{array}{ccc} v_2(\alpha^2+4yz) & = & k \\  \sqf(\alpha^2+4yz) & \equiv & 3\pmod{4}  \\  v_2(\alpha^2)< v_2(4yz) & & \end{array} \right)
=
\begin{cases} 
0 & \text{if } k \text{ is odd or } k = 0,\\
2^{-(3k/2+1)} &  \text{otherwise}.
\end{cases}
\]
\item 
\[
\Prob\left(\begin{array}{ccc} v_2(\alpha^2+4yz) & = & k \\  \sqf(\alpha^2+4yz) & \equiv & 2\pmod{4} \\ v_2(\alpha^2)< v_2(4yz) & &  \end{array} \right)
=
0.
\]
\end{enumerate}
\end{lem}

\begin{lem}\label{alpha big}
\begin{enumerate}
\item
\[
\Prob\left(\begin{array}{ccc} v_2(\alpha^2+4yz) & = & k \\  \sqf(\alpha^2+4yz) & \equiv & 1\pmod{8} \\ v_2(\alpha^2) > v_2(4yz) & &\end{array} \right)
=
\begin{cases} 
0 & \text{if } k \text{ is odd},\\
2^{-(3k/2+2)} &  \text{otherwise}.
\end{cases}
\]
\item
\[
\Prob\left(\begin{array}{ccc} v_2(\alpha^2+4yz) & = & k \\  \sqf(\alpha^2+4yz) & \equiv & 5\pmod{8} \\ v_2(\alpha^2) > v_2(4yz) & &\end{array}\right)
=
\begin{cases} 
0 & \text{if } k \text{ is odd},\\
2^{-(3k/2+2)} &  \text{otherwise}.
\end{cases}
\]

\item
\[
\Prob\left(\begin{array}{ccc} v_2(\alpha^2+4yz) & = & k \\  \sqf(\alpha^2+4yz) & \equiv & 3\pmod{4}\\ v_2(\alpha^2) > v_2(4yz) & & \end{array} \right)
=
\begin{cases} 
0 & \text{if } k \text{ is odd or } k=0,\\
2^{-(3k/2+1)} &  \text{otherwise}.
\end{cases}
\]

\item
\[
\Prob\left(\begin{array}{ccc} v_2(\alpha^2+4yz) & = & k \\  \sqf(\alpha^2+4yz) & \equiv & 2\pmod{4} \\ v_2(\alpha^2) > v_2(4yz) & & \end{array} \right)
=
\begin{cases} 
0 & \text{if } k \text{ is even or } k=1,\\
2^{-(3k/2-1/2)} &  \text{otherwise}.
\end{cases}
\]
\end{enumerate}
\end{lem}

\begin{lem}\label{alpha equal}
\begin{enumerate}
\item
\[
\Prob\left(\begin{array}{ccc} v_2(\alpha^2+4yz) & = & k \\  \sqf(\alpha^2+4yz) & \equiv & 1\pmod{8} \\ v_2(\alpha^2) = v_2(4yz) = 2\beta \end{array}\right)
=
\begin{cases} 
0 & \text{if } k \text{ is odd or } k \in \{0,2\},\\
2^{-(k+\beta+2)} &  \text{if } k \text{ is even and } 2 \le 2\beta \le k-1.
\end{cases}
\]
\item
\[
\Prob\left(\begin{array}{ccc} v_2(\alpha^2+4yz) & = & k \\  \sqf(\alpha^2+4yz) & \equiv & 5\pmod{8} \\ v_2(\alpha^2) = v_2(4yz) = 2\beta \end{array} \right)
=
\begin{cases} 
0 & \text{if } k \text{ is odd or } k \in \{0,2\},\\
2^{-(k+\beta+2)} &  \text{if } k \text{ is even and } 2 \le 2\beta \le k-1.
\end{cases}
\]

\item
\[
\Prob\left(\begin{array}{ccc} v_2(\alpha^2+4yz) & = & k \\  \sqf(\alpha^2+4yz) & \equiv & 3\pmod{4} \\ v_2(\alpha^2) = v_2(4yz) = 2\beta \end{array}  \right)
=
\begin{cases} 
0 & \text{if } k \text{ is odd or } k \in \{0,2\},\\
2^{-(k+\beta+1)} &  \text{if } k \text{ is even and } 2 \le 2\beta \le k-1.
\end{cases}
\]

\item
\[
\Prob\left(\begin{array}{ccc} v_2(\alpha^2+4yz) & = & k \\  \sqf(\alpha^2+4yz) & \equiv & 2\pmod{4} \\ v_2(\alpha^2) = v_2(4yz) = 2\beta  \end{array} \right)
=
\begin{cases} 
0 & \text{if } k \text{ is even or } k=1,\\
2^{-(k+\beta)} & \text{if } k \text{ is odd and } 2 \le 2\beta \le k-1.
\end{cases}
\]
\end{enumerate}
\end{lem}

Before proving these individual results, we see how they imply Theorem \ref{valuation_thm}.  We divide this argument into cases.  Combining these three lemmas, it is clear that 
\[
\Prob\left(\begin{array}{ccc} v_2(\alpha^2+4yz) & = &k \\ \sqf(\alpha^2+4yz) & \equiv & 1\pmod{8} \end{array}\right)  =   
\Prob\left(\begin{array}{ccc} v_2(\alpha^2+4yz) & = &k \\  \sqf(\alpha^2+4yz) & \equiv & 5\pmod{8} \end{array}\right),
\]
and that these probabilities are $0$ when $k$ is odd.  When $k = 0$ we have 
\[
\Prob\left(\begin{array}{ccc} v_2(\alpha^2+4yz) & = &0 \\ \sqf(\alpha^2+4yz) & \equiv & 1\pmod{8} \end{array}\right) 
= 
2^{-2} + 0 + 0,
\]
and when $k =2$ we have 
\[
\Prob\left(\begin{array}{ccc} v_2(\alpha^2+4yz) & = &2 \\ \sqf(\alpha^2+4yz) & \equiv & 1\pmod{8} \end{array}\right) 
= 
2^{-5} + 2^{-5} + 0 = 2^{-4}.
\]
Now suppose $k \ge 4$ is even.  Note that $\lfloor \frac{k-1}{2}\rfloor = k/2-1$. We have 
\begin{eqnarray*}
\Prob\left(\begin{array}{ccc} v_2(\alpha^2+4yz) & = &k \\ \sqf(\alpha^2+4yz) & \equiv & 1\pmod{8} \end{array}\right)  & = & 
2^{-(3k/2+2)} + 2^{-(3k/2+2)} + \sum_{\beta = 1}^{\lfloor\frac{k-1}{2}\rfloor} 2^{-(k+\beta+2)} \\ 
& = & 2^{-(3k/2+1)} + 2^{-(k+2)} \sum_{\beta=1}^{k/2- 1} 2^{-\beta}.
\end{eqnarray*}
We write 
\[
\sum_{\beta=1}^{k/2- 1} 2^{-\beta} = 2^{-1} \sum_{\beta=0}^{k/2- 2} 2^{-\beta} = 2^{-1}\left(2^{-(k/2-2)} + 2^{-(k/2-3)} + \cdots + 2^{-1} + 2^0\right).
\]
We see that 
\begin{eqnarray*}
\Prob\left(\begin{array}{ccc} v_2(\alpha^2+4yz) & = &k \\ \sqf(\alpha^2+4yz) & \equiv & 1\pmod{8} \end{array}\right) 
& = & 2^{-(3k/2+1)} + 2^{-(k+3)}(2^{-(k/2-2)}+ 2^{-(k/2-3)} + \cdots +2^{-1} +2^{0}) \\
& = &  2^{-(3k/2+1)} + 2^{-(3k/2+1)} +2^{-(3k/2+2)} + \cdots  + 2^{-(k+3)} \\
& = & 2^{-(k+2)}.
\end{eqnarray*}
We next consider the analogous computation for the case where $\sqf(\alpha^2+4yz)  \equiv  3\pmod{4}$.  Combining the lemmas above, we see that
\[
\Prob\left(\begin{array}{ccc} v_2(\alpha^2+4yz) & = &k \\ \sqf(\alpha^2+4yz) & \equiv & 3\pmod{4} \end{array}\right) = 0
\]
if $k$ is odd or $k = 0$.  Suppose that $k \ge 2$ is even.  We have
\begin{eqnarray*}
\Prob\left(\begin{array}{ccc} v_2(\alpha^2+4yz) & = &k \\ \sqf(\alpha^2+4yz) & \equiv & 3\pmod{4} \end{array}\right)  & = & 
2^{-(3k/2+1)} + 2^{-(3k/2+1)} + 2^{-(k+2)} \sum_{\beta=1}^{k/2- 1} 2^{-\beta},
\end{eqnarray*}
where for $k=2$ the empty sum in the final term is $0$.  Arguing as above, it is now clear that this sum is $2$ times the analogous one for $\sqf(\alpha^2+4yz)  \equiv  1\pmod{8}$.  

Finally, we consider the computation for the case where $\sqf(\alpha^2+4yz)  \equiv  2\pmod{4}$.  Combining the lemmas above, we see that
\[
\Prob\left(\begin{array}{ccc} v_2(\alpha^2+4yz) & = &k \\ \sqf(\alpha^2+4yz) & \equiv & 2\pmod{4} \end{array}\right) = 0
\]
if $k$ is even or $k =1$.  Suppose $k \ge 3$ is odd.  We have 
\[
\Prob\left(\begin{array}{ccc} v_2(\alpha^2+4yz) & = &k \\ \sqf(\alpha^2+4yz) & \equiv & 2\pmod{4} \end{array}\right) = 
2^{-(3k/2-1/2)} + \sum_{\beta=1}^{\frac{k-1}{2}} 2^{-(k+\beta)}.
\]
Note that 
\[
\sum_{\beta=1}^{\frac{k-1}{2}} 2^{-(k+\beta)} = 2^{-(k+1)} \sum_{\beta = 0}^{\frac{k-3}{2}} 2^{-\beta} = 2^{-(k+1)} 
\left(2^{-(k/2-3/2)} + 2^{-(k/2-4/2)} + \cdots + 2^{-1} + 2^0\right).
\]
This gives
\begin{eqnarray*}
\Prob\left(\begin{array}{ccc} v_2(\alpha^2+4yz) & = &k \\ \sqf(\alpha^2+4yz) & \equiv & 2\pmod{4} \end{array}\right) & = &
2^{-(3k/2-1/2)} + \left(2^{-(3k/2-1/2)} +2^{-(3k/2-3/2)} +\cdots + 2^{-(k+1)}\right) \\
& = &
2^{-k}. 
\end{eqnarray*}

We now prove the three lemmas.
\begin{proof}[Proof of Lemma \ref{alpha small}]
Suppose $v_2(\alpha^2) = k < v_2(4yz)$.  Therefore $k$ is even and $v_2(\alpha^2+4yz) = k$.  The probability that $v_2(\alpha) = k/2$ is $2^{-(k/2+1)}$.  The probability that $v_2(4yz) = 2 + v_2(z) > k$ is the probability that $v_2(z) \ge k-1$, which is $1$ if $k=0$ and is $2^{-(k-1)}$ if $k \ge 2$ is even.

Write $\alpha = 2^{k/2} u$ where $u \in \Z_2^*$ and $4yz = 2^{k+1} \nu$ where $\nu \in \Z_2$.  If $k = 0$ we must have $v_2(\nu) \ge 1$.  By varying $z$, we see that $\nu$ is a Haar random element of $2\Z_2$ when $k = 0$ and is a Haar random element of $\Z_2$ otherwise.

We have
\[
\sqf(\alpha^2 + 4yz) = \sqf(u^2 + 2\nu) \equiv u^2+2\nu \pmod{8}.
\]
Since $u^2 \equiv 1 \pmod{8}$ we see that
\[
u^2+2\nu  \equiv \begin{cases}
1 \pmod{8} & \text{ if } v_2(\nu) \ge 2,\\
5 \pmod{8} & \text{ if } v_2(\nu) = 1,\\
3 \pmod{4} & \text{ if } v_2(\nu) = 0.
\end{cases}
\]
We see that $\sqf(\alpha^2 + 4yz) \equiv 1 \pmod{8}$ if and only if $v_2(z) \ge 2k+1$, which happens with probability $2^{-(2k+1)}$, that  $\sqf(\alpha^2 + 4yz) \equiv 5 \pmod{8}$ if and only if $v_2(z) = 2k$, which happens with probability $2^{-(2k+1)}$, and that $\sqf(\alpha^2 + 4yz) \equiv 3 \pmod{4}$ if and only if $v_2(z) = 2k-1$, which happens with probability $2^{-2k}$ if $k \ge 2$ and probability $0$ if $k =0$.
\end{proof}

\begin{proof}[Proof of Lemma \ref{alpha big}]
Suppose $v_2(4yz) = 2+v_2(z) = k < v_2(\alpha^2)$.  So $v_2(\alpha^2+4yz) = k$.  The probability that $v_2(z) = k-2$ is $2^{-(k-1)}$ if $k \ge 2$ and is $0$ otherwise.  If $k\ge 2$ is even,
\[
\Prob(v_2(\alpha^2) > k) = \Prob(v_2(\alpha) \ge k/2 + 1) = 2^{-(k/2+1)},
\]
and if $k$ is odd, 
\[
\Prob(v_2(\alpha^2) > k) = \Prob(v_2(\alpha) \ge k/2 + 1/2) = 2^{-(k/2+1/2)}.
\]
Suppose $v_2(z) = k-2$ where $k \ge 2$.  We write $z = 2^{k-2} u$ where $u \in \Z_2^*$, so $4yz = uy 2^k$. Suppose $v_2(\alpha^2) > k$.  If $k$ is even, then $\alpha = \gamma 2^{k/2+1}$ where $\gamma \in \Z_2$ is not necessarily a unit.  In this case, $\sqf(\alpha^2 + 4yz) = \sqf(4 \gamma + yu)$.  For a fixed value of $z$, by varying $y$ we see that $yu$ is a Haar random element of $\Z_2^*$.  Therefore, the probability that $\sqf(\alpha^2 + 4yz) \equiv 3\pmod{4}$ is $1/2$, and the probability that $\sqf(\alpha^2 + 4yz) \equiv i \pmod{8}$ is $1/4$ for $i \in \{1,5\}$.  This completes the proof in the case that $k$ is even. If $k$ is odd, then $\alpha = \gamma 2^{k/2+1/2}$ where $\gamma \in \Z_2$ is not necessarily a unit.  In this case, $\sqf(\alpha^2 + 4yz) = \sqf(4 \gamma + 2yu) \equiv 2\mod{4}$.  This completes the proof when $k \ge 3$ is odd.
\end{proof}

\begin{proof}[Proof of Lemma \ref{alpha equal}]
Suppose that $v_2(\alpha^2) = v_2(4yz) = 2 + v_2(z)$.  Since $v_2(\alpha^2) = 2 v_2(\alpha)$, we must have $v_2(\alpha^2) = v_2(4yz) = 2 + v_2(z) = 2\beta$ with $\beta \ge 1$.

Suppose $v_2(\alpha) = \beta$ and write $\alpha = 2^\beta u$ where $u \in \Z_2^*$.  Suppose that $v_2(z) = 2\beta -2$ and write $z = 2^{2\beta-2} \nu$ where $\nu \in \Z_2^*$.  So $4yz = 2^{2\beta} y \nu$.  For a fixed value of $z$, varying $y$ shows that $y\nu$ is a Haar random element of $\Z_2^*$.

We have $v_2(\alpha^2+4yz) = 2\beta + v_2(u^2+y\nu)$.  Since $u^2$ and $y\nu$ are both units, $v_2(u^2+y\nu) \ge 1$ and we can write $u^2+y\nu = 2\delta$ where $\delta \in \Z_2$.  Since $y\nu$ is a Haar random element of $\Z_2^*$, we see that $\delta$ is a Haar random element of $\Z_2$.  Suppose that $k - 2\beta \ge 0$.  Therefore,
\[
\Prob(v_2(u^2+y\nu) = k-2\beta) = \begin{cases} 
0 & \text{if } k - 2\beta = 0\\ 2^{-(k-2\beta)} & \text{otherwise}.
\end{cases}
\]
We have 
\[
\sqf(\alpha^2+4yz) = \sqf(u^2+y\nu) = \sqf(2\delta).
\]
If $v_2(\delta)$ is even, then $\sqf(2\delta) \equiv 2 \pmod{4}$.  If $v_2(\delta)$ is odd, then for some nonnegative integer $r$ we have $2\delta = 2^{2 r} \delta'$, where $\delta' \in \Z_2^*$, and $\sqf(2\delta) = \sqf(\delta')$.  If we restrict to any particular value of $r$, since $\delta$ is a Haar random element of $\Z_2$, we see that $\delta'$ is a Haar random element of $\Z_2^*$.  In particular, the probability that $\sqf(\alpha^2+4yz) \equiv 3 \pmod{4}$ is $1/2$ and the probability that $\sqf(\alpha^2+4yz) \equiv i \pmod{8}$ is $1/4$ for $i \in \{1,5\}$.

The probability that $v_2(\alpha) = \beta$ is $2^{-(\beta+1)}$.  The probability that $v_2(z) = 2\beta -2$ is $2^{-(2\beta-1)}$ if $\beta \ge 1$ and is $0$ if $\beta = 0$.  We note that 
\[
2^{-(\beta+1)} 2^{-(2\beta-1)}2^{-(k-2\beta)} = 2^{-(k+\beta)}.
\]
Considering the different cases for $\sqf(\alpha^2+4yz)$ modulo $4$ and $8$ completes the proof.
\end{proof}

\section{Future Work} \label{preview_2}

If the groups $G$ and $G'$ are not as large as possible (\emph{i.e.},~do not have index 3 in $\GL_2(\Z_2)$), or if $G \not \simeq G'$, then the proportion $\mathcal{P}$ of anomalous primes might be quite different than $1/30$, as the following example shows.

\begin{exm}
Let $E$ be the elliptic curve  \href{https://www.lmfdb.org/EllipticCurve/Q/1200/e/5}{{\tt 1200e5}} and $E'$ the curve \href{https://www.lmfdb.org/EllipticCurve/Q/1200/e/2}{{\tt 1200e2}}.  Both mod 4 representations have order 4 and neither mod 8 representation contains $-I$.  By inspecting the 2-adic representations, one can check that the only possible defects of anomalous primes are $(3,2)$ and $(2,3)$. In fact, more is true.

If we look explicitly at the images of the mod 4 representations, we see
\begin{align*}
G(4) &= \left\{ \begin{pmatrix} \pm 1 & 0 \\ 0 & \pm 1 \end{pmatrix} \right\} \\
G'(4) &= \left\{ \begin{pmatrix}  1 & 0 \\ 0 & \pm 1 \end{pmatrix} ,  \begin{pmatrix}  -1 & 2 \\ 0 & \pm 1 \end{pmatrix} \right\}.
\end{align*}
If $p$ is anomalous, then using the fact that $p \equiv 1 \pmod{4}$ and that the 2-Sylow subgroups of $E(\F_p)$ and $E'(\F_p)$ are both $\Z/2\Z \times \Z/2\Z$, we must have $F \equiv  -I \pmod{4}$ and $F' \equiv \left(\begin{smallmatrix}  -1 & 2 \\ 0 & -1 \end{smallmatrix} \right) \pmod{4}$. Therefore, every anomalous prime has defect $(3,2)$ and by the Chebotarev density theorem this is exactly 1/4 of all primes.
\end{exm}

In a forthcoming paper \cite{rnt}, we take up the problem of determining all possible values of $\mathcal{P}$, for all pairs of rationally 2-isogenous elliptic curves over $\Q$, including the case where $E$ and $E'$ have CM.  What makes this a finite task is that 
\begin{enumerate}
\item all images of 2-adic representations have been classified (\cite{rzb} for the non-CM case and \cite{alvaro} for the CM case), and
\item all isogeny-torsion graphs over $\Q$ have been classified in \cite{chiloyan} and \cite{chil-alvaro}. 
\end{enumerate}
There are additional consequences for the isogeny volcanoes attached to these curves that we explore as well, including how the torsion point fields $\Q(E[2^m])$ and $\Q(E'[2^m])$ are ``entangled''.   For example, we are able to show the following two results.
\begin{itemize}
\item If $\Q(E[2]) = \Q(E'[2])$ then $G$ and $G'$ must each have index greater than $3$ in $\GL_2(\Z_2)$.
\item If there are no primes of defect $(m+1,m)$ then we must have $\Q(E[2^m]) = \Q(E'[2^m])$ and $\Q(x(E[2^m])) = \Q(x(E'[2^m]))$.
\end{itemize}
We explore the consequences of these and similar results for anomalous primes.

\appendix

\section{Sample Calculations} \label{calculations}

Here we present some corroborating evidence for Theorem \ref{mainthm2} which served as the impetus for this project.  In the table below we present 15 pairs of curves whose 2-adic images have index 3 in $\GL_2(\Zl)$  and list the number of anomalous primes up to $2^{30}$.  The proportions listed are the number of anomalous primes divided by $\pi(2^{30}) = 54400028$.  One can see the 1/30 proportion very clearly emerging in the data.  These calculations were performed on \textsf{Magma} \cite{magma} by Andrew Sutherland and we thank him for allowing us to include these data in this paper.

\bigskip

\begin{center}
\begin{tabular}{|l|l|l|l|}
\hline
$E$ & $E'$ & Anomalous &  Proportion   \\
\hline
69a1 & 69a2 & 1814517 &0.033355075\\
\hline
77c1 & 77c2 & 1812315 & 0.033314597 \\
\hline
84b1 & 84b2 & 1813293 & 0.033332575 \\
\hline
99a1 & 99a2 & 1812977 & 0.033326766 \\
\hline
99c1 & 99c2 & 1812977 & 0.033326766 \\
\hline
132a1 & 132a2 & 1812966 & 0.033326564 \\
\hline
132b1 & 132b2 & 1812959 & 0.033326435 \\
\hline
138a1 & 138a2 & 1813813 &  0.033342134\\
\hline
141b1 & 142b2 & 1812863 & 0.033324670 \\
\hline
154a1 & 154a2 & 1812080 & 0.033310277 \\
\hline
154c1 & 154c2 & 1813344 & 0.033333512 \\
\hline
155b1 & 155b2 & 1813606 & 0.033338328 \\
\hline
156a1 & 156a2 & 1813340 & 0.033333439 \\
\hline
10608y1 & 10608y2 & 1812615 & 0.033320112 \\
\hline
10608j1 & 10608j2 & 1814206 & 0.033349358 \\
\hline
\end{tabular}
\end{center}

\bigskip

We also include some data for the pair $(E,E')$ of rationally 2-isogenous elliptic curves over $\Q$ where $E$ has \texttt{LMFDB} label \href{https://www.lmfdb.org/EllipticCurve/Q/69/a/2}{{\tt 69a2}} and $E'$ has label \href{https://www.lmfdb.org/EllipticCurve/Q/69/a/1}{{\tt 69a1}}.  We computed that there were 42298 anomalous primes less than $2\cdot 10^7$, a proportion of approximately $0.0333$ among all primes.  They are distributed by defect as follows: 
\begin{center}
\begin{tabular}{|ll|lr|}
\hline
(3,2): & 19821 & &\\
(2,3): & 19831 & \textbf{Total:} &39652 \\
\hline
(4,3): & 1264 &&\\
(3,4): & 1205& \textbf{Total:}&  2469 \\
\hline
(5,4): &84 &&\\
(4,5): & 86& \textbf{Total:}& 170 \\
\hline
(6,5): &3& &\\
(5,6): &4& \textbf{Total:}& 7 \\
\hline
\end{tabular}
\end{center}
\bigskip

We now look more closely at the 19821 of these anomalous primes with defect $(3,2)$ and divide them up into rows based on $\disc \mathcal{O}_0 \pmod{8}$ and columns based on the height of $V_p$, the isogeny volcano associated to $(E,E')$:
\begin{center}
\begin{tabular}{|c|c|c|c|c|c|c|}
\hline
$\disc \mathcal{O}_0\pmod{8}\ \backslash\ h(V_p)$ & 2 & 3 & 4 & 5 & 6 & $\ge$ 7 \\ 
\hline
1 & 4930 & 1279 & 322 & 76 & 22 & 7 \\
\hline
5 & 5024 & 1225 & 308 & 82 & 31 & 4 \\
\hline
0 & 2501 & 570 & 168 & 45 & 10 & 3 \\
\hline
4 & 2363 & 628 & 172 & 38 & 8 & 5 \\
\hline
\end{tabular}
\end{center}

\bigskip

We give an analogous table for the 1264 of these anomalous primes with defect $(4,3)$:
\begin{center}
\begin{tabular}{|c|c|c|c|c|c|}
\hline
$\disc \mathcal{O}_0\pmod{8}\ \backslash\ h(V_p)$ & 3 & 4 & 5 & 6 & $\ge$ 7 \\ 
\hline
1 & 305 & 73 & 20 & 5 & 2  \\
\hline
5 & 318 & 85 & 18 & 5 & 1 \\
\hline
0 & 155 & 40 & 13 & 5 & 0 \\
\hline
4 & 158 & 28 & 9 & 2 & 2 \\
\hline
\end{tabular}
\end{center}

In both cases observe that the values decrease roughly by a factor of 4 as we move along a row, as predicted by Conjecture \ref{volcano_conj}.

\vfill

\end{document}